\documentclass[pdflatex,sn-mathphys-num]{sn-jnl}

\usepackage{graphicx}%
\usepackage{subcaption}
\usepackage{multirow}%
\usepackage{amsmath,amssymb,amsfonts}%
\usepackage{amsthm}%
\usepackage[title]{appendix}%
\usepackage{xcolor}%
\usepackage{textcomp}%
\usepackage{manyfoot}%
\usepackage{booktabs}%
\usepackage{algorithm}%
\usepackage{algorithmicx}%
\usepackage{algpseudocode}%
\usepackage{listings}%
\usepackage{mathdots}
\usepackage{natbib}

\theoremstyle{thmstyleone}%
\newtheorem{theorem}{Theorem}
\newtheorem{proposition}[theorem]{Proposition}%
\newtheorem{corollary}[theorem]{Corollary}
\newtheorem{problem}{Problem}
\newtheorem{example}{Example}%
\newtheorem{remark}[theorem]{Remark}%
\newtheorem{lemma}[theorem]{Lemma}
\theoremstyle{thmstylethree}%

\begin{document}

\title[Article Title]{Operator Inequalities and Several Characterizations of the $\lambda$-Mean Transform}


\author[1]{\fnm{Bikram} \sur{Das}}\email{dasbikram642@gmail.com}
\author[2]{\fnm{Goutam} \sur{Biswas}}\email{biswasgoutam047@gmail.com}

\author*[3]{\fnm{Chandal} \sur{ Nahak}}\email{cnahak@maths.iitkgp.ac.in}
\equalcont{The author contributed equally to this work.}

\affil[1,2,3]{\orgdiv{Department of Mathematics}, \orgname{Indian Institute of Technology Kharagpur}, \orgaddress{\city{Kharagpur}, \postcode{721302}, \state{West Bengal}, \country{India}}}


\abstract{
We broaden Buzano-type inequalities to provide novel numerical radius bounds for operators of the type $AXB$, thereby generalizing the results obtained by Sababheh et al. For the $\lambda$-mean transform $M_\lambda(T)$, we provide a counterexample demonstrating that $r_\sigma(M_\lambda(T)) \le r_\sigma(T)$ fails to hold in general for $\lambda \in (0, 1)$, establish that $(r_\omega(M_\lambda(T)))^n$ and $r_\omega(T^n)$ are typically incomparable for $n \ge 2$, and confirm that $M_\lambda(T^*) = (M_\lambda(T))^*$ is valid for $\lambda \in [0, 1)$ if and only if $T$ is a member of a newly established $\sigma$-class. Furthermore, we examine the transformation characteristics of $T$ and the tensor products $T \otimes S$, refine Zamani’s inequalities, and unify operator modulus bounds $\vert{}\widetilde{T}\vert{} \le \vert{}\widehat{T}\vert{} \le \vert{}T\vert{}$. In this application, we demonstrate that the conditions for norm preservation, $\Vert{}\widetilde{T}\Vert{} = \Vert{}T\Vert{}$ and $\Vert{}M_\lambda(T)\Vert{} = \Vert{}T\Vert{}$, are equivalent to the statement $\Vert{}T^2\Vert{} = \Vert{}T\Vert{}^2$, and we offer precise norm estimates for $2 \times 2$ off-diagonal block operator matrices under $\lambda$-mean transformation.
}




\maketitle

\section{Introduction}
A fundamental component of contemporary functional analysis is the theory of bounded linear operators on Hilbert spaces. The operator norm and the numerical radius are among the fundamental tools in this framework. These concepts are not only essential in operator theory, but they also have a significant impact on a variety of fields, including mathematical physics, control theory, quantum information theory, and numerical linear algebra.  Furthermore, they encode certain geometric and spectral features of linear transformations, which are of great importance.

Let $\mathcal{H}$ be a complex Hilbert space equipped with an inner product $\langle \cdot , \cdot \rangle$ and the corresponding norm defined by $\|x\| = \sqrt{\langle x,x \rangle}$. 
A unit vector in $\mathcal{H}$ is an element $x\in \mathcal{H}$ with $\|x\|=1$. 
$\mathbb{B}(\mathcal{H})$ represents the $C^*$algebra of all bounded linear operators acting on $\mathcal{H}$.

For an operator $S \in \mathbb{B}(\mathcal{H})$, the operator norm is defined by
$$
\|S\| 
= \sup_{\|x\|=1} \|Sx\|
= \sup_{\|x\|=\|y\|=1} |\langle Sx , y \rangle|.
$$

The numerical range of $S$ is represented by $W(S)$, which is the subset of the  complex plane  provided by $$ W(S) = \left\{ \langle Sx, x \rangle: x \in \mathcal{H}, \ \|x\| = 1 \right\}. $$
The corresponding numerical radius is given by
$$
r_{\omega}(S) 
= \sup_{\|x\|=1} |\langle Sx , x \rangle|.
$$

The spectral radius $S$ is defined by
$$
r_{\sigma}(S) 
= \sup_{\lambda \in \sigma(S)} |\lambda|,
$$ 
where the spectrum of $S$ is
$$
\sigma(S) 
= \left\{ \lambda \in \mathbb{C} : S - \lambda I \text{ is not invertible in } 
\mathbb{B}(\mathcal{H}) \right\}.
$$

An operator $S \in \mathbb{B}(\mathcal{H})$ is called positive, denoted by $S \geq 0$, if it is self-adjoint ($S = S^*$) and satisfies
$$
\langle Sx, x \rangle \geq 0 
\quad \text{for every } x \in \mathcal{H}.
$$
We denote by $\mathbb{B}(\mathcal{H})^{+}$ the set of all positive operators on $\mathcal{H}$.

The absolute value (or modulus) of $S$ and its adjoint $S^*$ in $\mathbb{B}(\mathcal{H})$ are defined, respectively, by  $$ |S| = (S^*S)^{1/2}, \qquad |S^*| = (SS^*)^{1/2},$$  where $(\cdot)^{1/2}$ signifies the unique positive square root of a positive operator. 

The numerical radius provides a norm on $\mathbb{B} (\mathcal{H})$, and it adheres to the inequality: 
\begin{equation}\label{(1)} \frac{\|S\|}{2}\le{r_{\omega}(S)}\le{\|S\|}\tag{1} . \end{equation}

If $S^2=0,$ then $r_{\omega}(S)=\frac{1}{2}\|S\|$ and if $S$ is self-adjoint, then   $r_{\omega}(S)=\|S\|.$

The power inequality is a significant attribute of the numerical radius, signifying that $r_{\omega}(S^{n}) \le (r_{\omega}(S))^n,$ for all $n \in\mathbb{N}.$

The reverse power inequality, $r_{\omega}(S^{n}) \ge (r_{\omega}(S))^n$, generally does not hold.
 For a matrix
 $S=$
 $\begin{bmatrix}
 0 & 0\\ 
 1 & 0
 \end{bmatrix}$, we have $r_{\omega}(S)=\frac {1}{2}$ but $r_{\omega}(S^{n})=0, \ \forall n\ge 2.$

Many researchers have determined a range of upper and lower bounds to improve the inequality (\ref{(1)}) over the years, see e.g. [\citenum{kittaneh2003numerical,kittaneh2005numerical,[abu2015upper],el2007numerical,dragomir2008power,bhunia2021furtherance} ].

The Cauchy–Schwarz inequality is one of the most fundamental and extensively used inequalities in functional analysis, operator theory, and inner product spaces. It asserts that 
\begin{equation}\label{(2)}
|\langle x,y\rangle| \le \|x\|\,\|y\|, 
\qquad x,y \in \mathcal{H}. 
\tag{2}
\end{equation}
Due to its key importance in mathematical physics, quantum mechanics and numerical analysis, finding sharp refinements and generalizations of (\ref{(2)}) has been the topic of considerable research for decades.

Buzano [\citenum{buzano1974generalizzazione}] presented a refinement of (\ref{(2)}), demonstrating that \begin{equation}\label{(3)}
|\langle x,e\rangle \langle e,y\rangle|
\le \frac{1}{2}\big( \|x\|\|y\| + |\langle x,y\rangle| \big), 
\tag{3}
\end{equation}
for all $x,y,e \in \mathcal{H}$ with $\|e\|=1$.

For an arbitrary vector $z\in \mathcal{H}$, Buzano's inequality (\ref{(3)}) can be generalized in a natural way as follows:
\begin{equation}\label{4}
|\langle x,z\rangle \langle z,y\rangle|
\le \frac{\|z\|^2}{2}\big( \|x\|\|y\| + |\langle x,y\rangle| \big). 
\tag{4}
\end{equation}

In recent years, the attention has been paid to vector-norm inequalities in their bounded linear operator equivalents in the algebra $\mathcal{B}(\mathcal{H})$. This gap was filled by Dragomir [\citenum{dragomir2017buzano}, Corollary 1] who proved a notable operator-based analogue closely connected to the Buzano's inequality (\ref{(3)}):
\begin{equation}\label{5}
|\langle Sx, Sy\rangle|
\le \frac{\|S\|^2}{2}\big( |\langle x,y\rangle| + \|x\|\|y\| \big), 
\qquad S \in \mathcal{B}(\mathcal{H}). 
\tag{5}
\end{equation}

Sababheh et al.  [\citenum{sababheh2025inner}, Lemma 2.1] extended these constraints to multi-operator settings, building on the framework presented by Dragomir. Especially, they proved a new estimate using a positive operator $X$ and arbitrary operators $A, B\in \mathcal{B}(\mathcal{H})$ by means of (\ref{5}):

\begin{equation}\label{6}
\left| \langle AXB x , y \rangle \right|
\le 
\frac{\|X\|}{2}
\left(
|\langle ABx , y \rangle|
+
\|Bx\|\,\|A^{*}y\|
\right)
\tag{6}
\end{equation}
Using these operator-space variations, Sababheh et al. [\citenum{sababheh2025inner}, Theorem 2.4] derived new upper bounds for the numerical radius $r_{\omega}(ST)$ of product operators, demonstrating that: 
\begin{equation}\label{7}
r_{\omega}(ST)
\le
\frac{\|S\|^{1/2}}{2}
\|\,|S^*|+|T|^2\,\|
\tag{7}
\end{equation}
and
\begin{equation}\label{8}
r_{\omega}(ST)
\le
\frac{\|S\|^{1/2}}{2}
\left[
\frac12\|\,|S^*|+|T|^{2}\,\|
+\|S\|^{1/2}\;\|T\|
\right].
\tag{8}
\end{equation}

Our main goal with this work is to get rid of static norm inequalities by adding a free complex parameter $\lambda \in \mathbb{C}$ to the basic estimates of operator configurations with the form $AXB$, where $X$ is a positive operator.
Based on our generalized estimates for the inner product, we give a unified framework for the estimation of the numerical radius $r_{\omega}(AXB)$.

We now proceed to our subsequent motivation, namely, the investigation of the $\lambda-$mean transformation of a Hilbert space operator.
To provide context, we quickly review a number of popular transforms related to bounded linear operators on a Hilbert space.  
Let $T$ be an element of $\mathcal{B}(\mathcal{H})$, and denote its canonical polar decomposition as
$T = U|T|,$ where $U$ is a partial isometry and $|T| = (T^*T)^{1/2}$.

Among the most extensively researched transforms is the Aluthge transform $\widetilde{T},$ defined by
$\widetilde{T} = |T|^{1/2} U |T|^{1/2}.$
Another classical transform is the Duggal transform $T^{D},$ given  by
$T^{D} = |T| U.$
The mean transform $\widehat{T}$ of $T$ is defined by
$\widehat{T} = \frac12 (T + T^{D}).$
Extending these constructions yields the generalized mean transform $\widehat{T}(t)$, defined for each $t \in [0,1]$ by 
$$\widehat{T}(t)
=
\frac{1}{2}
\left(
|T|^{t} U |T|^{1-t}
+
|T|^{1-t} U |T|^{t}
\right).$$

Clearly,
$\widehat{T}(0) = \widehat{T}
\quad \text{and} \quad
\widehat{T}\!\left(\frac12\right) = \widetilde{T}.$

In 2021, Zamani [\citenum{zamani2021extension}] introduced the $\lambda$–mean transform $M_{\lambda}(T)$ 
of $T.$  It is defined by
$$
M_{\lambda}(T)
=
\lambda T
+
(1-\lambda) T^{D}, \qquad \lambda \in [0,1].
$$
In particular,
$
M_{0}(T) = T^{D}
\quad \text{and} \quad
M_{1/2}(T) = \widehat{T}.
$

In [\citenum{zamani2021extension}], Zamani  established two-sided bounds for the norm of the $\lambda$–mean transform:
\begin{equation}\label{9}
2\sqrt{\lambda - \lambda^2}\,\|\widetilde{T}\|
\le
\|M_{\lambda}(T)\|
\le
\lambda\|T\| + (1-\lambda)\|T^{D}\|.
\tag{9}
\end{equation}

In particular,
\begin{equation}\label{10}
2\sqrt{\lambda - \lambda^2}\, r(T)
\le
\|M_{\lambda}(T)\|
\le
\|T\|,
\tag{10}
\end{equation}
where $T\in \mathcal{B}(\mathcal{H})$ and $\lambda \in [0,1].$

In the same paper [\citenum{zamani2021extension}], Zamani improved the second inequality of (\ref{9}) by demonstrating that if  $T\in \mathcal{B}(\mathcal{H})$ and $\lambda \in [0,1],$
then 
\begin{equation}\label{11}
\|M_{\lambda}(T)\|
\le
\frac{
\|\lambda| T |+ (1-\lambda)|T^{D}|\,\|
+
\|\lambda |T^{*} |+ (1-\lambda)|(T^{D})^{*}|\,\|
}{2}
\le
\lambda\|T\| + (1-\lambda)\|T^{D}\|.
\tag{11}
\end{equation}

In particular,
\begin{equation}\label{12}
\|\widehat{T}\|
\le
\frac{
\|\,|T| + |T^{D}|\,\|
+
\|\,|T^{*}| + |(T^{D})^{*}|\,\|
}{4}
\le
\|T\|.
\tag{12}
\end{equation}

Since $r_{\sigma}(.)$ is commutative, we have, for $\lambda = 0,$  $M_0(T) = T_D$, $r_{\sigma}(T_D) = r_{\sigma}(|T|V) = r_{\sigma}(V|T|) = r_{\sigma}(T).$ For $\lambda = 1,$ $M_1(T) = T$, $r_{\sigma}(M_{1}(T)) = r_{\sigma}(T).$ 

For the quasi-normality of $T,$ $r_{\sigma}(M_{\lambda}(T)) = r_{\sigma}(T)$ holds for $\lambda \in [0,1].$ This observation leads to a natural question:
\begin{problem}\label{problem1}
    Does an inequality analogous to $r_{\omega}(M_\lambda(T)) \le r_{\omega}(T)$ hold for the spectral radius; that is, does $r_{\sigma}(M_\lambda(T)) \le r_{\sigma}(T)$ hold for all $\lambda \in (0,1)$?
\end{problem}

It is clear that for $\lambda = 1,$ $M_1(T) = T$, $M_{1}(T^*) = (M_{1}(T))^*.$ It raises immediately the following question:
\begin{problem}\label{problem2}
Under what condition does $M_\lambda(T^*) = (M_\lambda(T))^*$ holds for $\lambda \in [0,1)$?
\end{problem}

We have for $n \in \mathbb{N},$ $r_{\omega}(T^n) \le r_{\omega}^n(T)$ and $r_{\omega}(M_\lambda(T)) \le r_{\omega}(T)$. 
That implies $(r_{\omega}(M_\lambda(T)))^n \le (r_{\omega}(T))^n,$  for $n \in \mathbb{N}.$

For $\lambda = 1$ or $T$ is quasi-normal, we have  $M_\lambda(T) = T$. In this case, $ r_{\omega}(T^n) \le ( r_{\omega}(M_\lambda(T)))^n$ holds for $n \in \mathbb{N}$.

For $n=1$, $r_{\omega}(M_\lambda(T))\le r_{\omega}(T) $ holds for $\lambda \in [0, 1]$. This motivates the next question:
\begin{problem}\label{problem3}
Is there any relation between $(r_{\omega}(M_\lambda(T)))^n$ and $r_{\omega}(T^n)$ for $\lambda \in [0,1)$ and $n \ge 2$?
\end{problem}

There has been a lot of interaction between operator mean transforms and Aluthge type  transforms for their involvement in refining norm and numerical radius inequalities.

Notice that $\|\widetilde{T}\|=\| |T|^{1/2} U |T|^{1/2}\|\le \|U\|\; \|\,|T|^{\frac{1}{2}}\,\|^2\le\|T\|.$
In 2019, Chabbabi et al. [\citenum{chabbabi2019mean}] showed that if $T \in \mathcal{B}(\mathcal{H}),$ then
\begin{equation}\label{13}\|\widetilde{T}\|\le\|\widehat{T}\|\le\|T\|.\tag{13}\end{equation}

Currently, in Altwaijry et al. [\citenum{altwaijry2024new}, Theorem 3.2] the authors established that if $T \in \mathcal{B}(\mathcal{H})$ and $\lambda \in \left[0,\frac12\right],$ then
\begin{equation}\label{14}
\|\widetilde{T}\|
\le
\|\widehat{T}(\lambda)\|
\le
\frac{1}{2}
\left(
\|T^{D}\|^{\lambda}\|T\|^{1-\lambda}
+
\|T^{D}\|^{1-\lambda}\|T\|^{\lambda}
\right)
\le
\|T\|\tag{14}
\end{equation}

and
\begin{equation}\label{15}
\|\widetilde{T}\|
\le
\|\widehat{T}(\lambda)\|
\le
2\lambda \|\widetilde{T}\|
+
(1-2\lambda)\|\widehat{T}\|
\le
\|\widehat{T}\|
\le
\frac12\left(\|T^{D}\|+\|T\|\right)
\le
\|T\|.\tag{15}
\end{equation}

In particular, $\|T^{D}\|\le \|T\|$ and if $\lambda=0$, then
\begin{equation}\label{16}
\|\widetilde{T}\|
\le
\|\widehat{T}\|
\le
\frac12\left(\|T^{D}\|+\|T\|\right)
\le
\|T\|.
\tag{16}
\end{equation}
Even though Zamani, Chabbabi, and Altwaijry made a lot of progress, the complicated relationship between the Aluthge, Duggal, and $\lambda$-mean transforms shows that the present inequalities are still not complete. Because of these recent progresses, the goal of this paper is to make the norm inequalities that come with these parameterized transforms even more general and better. By building on Zamani and Altwaijry's work, we hope to find more accurate estimates and create a more unified way to bound the $\lambda$-mean transform.

\subsection{Organization of the paper}
The structure of this paper is as follows:

In Section \ref{section2}, by extending Buzano-type inequalities to derive revised bounds for operators of the form $AXB$, we are able to obtain  numerical radius estimates that generalize the findings that Sababheh et al. [\citenum{sababheh2025inner}]  had previously discovered.  

In Section \ref{section3}, we methodically tackle and deliver comprehensive answers to all three issues presented earlier. We proceed to obtain stronger upper bounds for the $\lambda$–mean transform $M_{\lambda}(T)=\lambda T+(1-\lambda)T^{D}$ in regard to $|T|$ and $|T^*|$, thereby strengthening Zamani's inequalities (\ref{9}), (\ref{10}). We also offer an enhanced norm estimate for the $\lambda$-mean transform via an integral numerical-radius representation of a related operator matrix. We provide a necessary and sufficient criterion for the norm-preserving characteristic of the $\lambda$-mean transform.
\section{Numerical radius inequalities} \label{section2}

The following lemmas are crucial for our purpose.

\begin{lemma} \label{Lemma8}

[\citenum{bhatia1993more}]
 Let $A,B \in \mathcal{B}(\mathcal{H}).$  Then
 \begin{align*}
  \|AB^*\|\le\frac{1}{2}\|A^*A+B^*B\|.   
 \end{align*}
\end{lemma}

\begin{lemma}\label{Lemma9}

[\citenum{kittaneh1988notes}] Let $P$ $\in{\mathcal{B(H)}}$ and $x,y$ $\in{\mathcal{H}}$. If $f,g$ are non-negative continuous functions on $[0,\infty)$ which are satisfying the relation $f(c)g(c)=c$ $(c\in[0,\infty))$, then
$$\Big|\langle Px,y\rangle\Big|^2\le\Big\langle |P|x,x\Big\rangle \Big\langle |P^*|y,y \Big\rangle$$ and more generally
$$\Big|\langle Px,y\rangle\Big|^2\le\Big\langle f^2(|P|)x,x\Big\rangle \Big\langle g^2(|P^*|)y,y\Big \rangle.$$
\end{lemma}

We commence by extending the bound established by Sababheh et al.[\citenum{sababheh2025inner}]. We introduce an arbitrary complex parameter $\lambda$ and obtain a more general class of Buzano-type inequalities for operators of the kind $AXB$. The known result follows immediately as a particular case, whereas the parameterized version demonstrated here is more flexible and serves as a basis for various better numerical radius inequalities obtained later in this study.

\begin{theorem}\label{Theorem 1}
    Let $A,B \in \mathcal{B}(\mathcal{H})$ and $X \in \mathcal{B}(\mathcal{H})^{+}$. Then for $\lambda \in \mathbb{C}$,
\begin{equation}\label{17}
\left| \langle AXB x , y \rangle \right|
\le 
\|X\|
\left[
\max\{|\lambda|, |1-\lambda|\}
\|Bx\|\,\|A^{*}y\|
+
|\lambda|\, |\langle ABx , y \rangle|
\right].
\tag{17}\end{equation}
In particular,
\begin{equation}\label{18}
\left| \langle AXB x , y \rangle \right|
\le 
\frac{\|X\|}{2}
\left(
|\langle ABx , y \rangle|
+
\|Bx\|\,\|A^{*}y\|
\right).
\tag{18}\end{equation}
\end{theorem}

\begin{proof} : Since inequality (\ref{18}) is obtained from (\ref{17}) by putting $\lambda=\frac{1}{2},$ it is sufficient to prove (\ref{17}).
If $X=O$ then the inequality (\ref{17}) is trivially true.

Assume first that $X$ is a non-zero positive contraction operator.  
For any $\lambda \in \mathbb{C}$, we have
\begin{align*}
    |\langle Xx , y \rangle|
\le
|\langle (X-\lambda I)x , y \rangle|
+
|\lambda|\,|\langle x , y \rangle|.
\end{align*}
By Cauchy-Schwarz inequality, we get
\begin{equation}\label{19}
|\langle Xx , y \rangle|
\le
\|X-\lambda I\|\,\|x\|\,\|y\|
+
|\lambda|\,|\langle x , y \rangle|.\tag{19}
\end{equation}

Since $\sigma(X) \subset [0,1]$, we obtain
\[
\|X-\lambda I\|
=
\sup_{k \in \sigma(X)} |k-\lambda|
\le
\max\{|\lambda|,|1-\lambda|\}.
\]
Now from \eqref{19}, we get
\begin{equation}\label{20}
|\langle Xx , y \rangle|
\le
\max\{|\lambda|,|1-\lambda|\}\,\|x\|\,\|y\|
+
|\lambda|\,|\langle x , y \rangle|.\tag{20}
\end{equation}

Define the operator $\widetilde{X}$ by
\[
\widetilde{X}=\frac{X}{\|X\|},\,
\text{where $X\ne O$ is a positive operator.}\]
It is easy to see that $\widetilde{X}$ is a positive contraction.  
Replacing $X$ by $\widetilde{X}$ in \eqref{20}, we infer that
\[
|\langle Xx , y \rangle|
\le
\|X\|
\Big[
\max\{|\lambda|,|1-\lambda|\}\,\|x\|\,\|y\|
+
|\lambda|\,|\langle x , y \rangle|
\Big].
\]

Finally, replacing $x$ and $y$ by $Bx$ and $A^{*}y$, respectively, we obtain
\[
|\langle AXB x , y \rangle|
\le
\|X\|
\Big[
\max\{|\lambda|,|1-\lambda|\}\,\|Bx\|\,\|A^{*}y\|
+
|\lambda|\,|\langle ABx , y \rangle|
\Big],
\]
which completes the proof.
\end{proof}

The subsequent outcome offers a cohesive framework for estimating the numerical radius of $AXB$ when $X$ is positive. The resulting inequalities include a free parameter $\lambda$ and produce adaptable bounds given in terms of $r_{\omega}(AB)$ and operator norms.

\begin{theorem}\label{Theorem2}
Let $A,X,B \in \mathcal{B}(\mathcal{H})$ and let $X$ be positive. Then
\begin{equation}\label{21}
r_{\omega}(AXB)
\le
\|X\|
\left[
\frac{\max\{|\lambda|,|1-\lambda|\}}{2}
\|\,|A^*|^{2}+|B|^{2}\|
+
|\lambda|\,r_{\omega}(AB)
\right]\tag{21}
\end{equation}
and
\begin{equation}\label{22}
r_{\omega}(AXB)
\le
\|X\|
\left[
|\lambda|\,r_{\omega}(AB)
+
\frac{\max\{|\lambda|,|1-\lambda|\}}{2}
\left(
\|A\|\|B\|
+
\|BA\|
\right)
\right].\tag{22}
\end{equation}
\end{theorem}
\begin{proof}: 
 The proof adheres to the methodology outlined in Theorem 2.4 of [\citenum{sababheh2025inner}]. In view of the previous result, the desired bounds are derived by straightforward adaptations, and hence, we omit the details.   
\end{proof}

From Theorem \ref{Theorem2}, we immediately get the following corollary:
\begin{corollary}\label{Corollary9}
Let $S,T\in \mathcal{B}(\mathcal{H}).$ Then for any $\lambda\in{\mathbb{C}}$ and $t\in[0,1],$ we have

  \begin{align*}
(a)\quad &r_{\omega}(ST)
\le\frac{\max\{|\lambda|,|1-\lambda|\} +|\lambda|}{2} \;\|S\|^{t}\;
\|\,|S^*|^{2(1-t)}+|T|^2\,\|
\qquad\text{and}
\\&r_{\omega}(ST)
\le
\frac{\|S\|^{t}}{2}
\Big[
|\lambda|\;
\|\,|S^*|^{2(1-t)}+|T|^2\,\|
+2\max\{|\lambda|,|1-\lambda|\}\;
\|S\|^{1-t}\;\|T\|
\Big].
\end{align*}

\begin{align*} (b) &\quad r_{\omega}(ST)
\le
\frac{\max\{|\lambda|,|1-\lambda|\} +|\lambda|}{2} \;\|S\|^{1-t}\;
\|\,|S^*|^{2t}+|T|^2\,\|
\qquad\text{and}
\\&
r_{\omega}(ST)
\le
\frac{\|S\|^{1-t}}{2}
\Big[
|\lambda|\;
\|\,|S^*|^{2t}+|T|^2\,\|
+2\max\{|\lambda|,|1-\lambda|\}
\;\|S\|^{t}\;\|T\|
\Big].
\end{align*}
\begin{align*}
(c)\quad &r_{\omega}(ST)
\le
\frac{\max\{|\lambda|,|1-\lambda|\} +|\lambda|}{2} \;\|T\|^{t}\;
\|\,|T|^{2(1-t)}+|S^*|^2\,\|
\qquad\text{and}
\\&
r_{\omega}(ST)
\le
\frac{\|T\|^{t}}{2}
\Big[
|\lambda|\;
\|\,|T|^{2(1-t)}+|S^*|^2\,\|
+2\max\{|\lambda|,|1-\lambda|\}\;
\|T\|^{1-t}\;\|S\|
\Big].
\end{align*}

\begin{align*} (d) \quad&r_{\omega}(ST)
\le
\frac{\max\{|\lambda|,|1-\lambda|\} +|\lambda|}{2} \;\|T\|^{1-t}\;
\|\,|T|^{2t}+|S^*|^2\,\|
\qquad\text{and}
\\&
r_{\omega}(ST)
\le
\frac{\|T\|^{1-t}}{2}
\Big[
|\lambda|\;
\|\,|T|^{2t}+|S^*|^2\,\|
+2\max\{|\lambda|,|1-\lambda|\}
\;\|T\|^{t}\;\|S\|
\Big].
\end{align*}
\end{corollary}

\begin{proof}:
  Let $S = U|S|$ be the polar decomposition of $S$. We have the following fact: 
  
  $U|A|^{p}=|A^*|^{p}U,\,U|A|^{p}U^{*}=|A^*|^{p}, 
\,U^{*}|A^*|^{p}U=|A|^{p},
\, (p>0),$ for any $A \in \mathcal{B}(\mathcal{H})$ 
and unitary operator $U.$

Setting $A = U|S|^{1-t}$, $X = |S|^{t}$ and $B = T$ in Theorem \ref{Theorem2}, we obtain
\begin{align}
    \label{23}
r_{\omega}(ST)
&\le
\|S\|^{t}
\Bigg[
\frac{\max\{|\lambda|,|1-\lambda|\}}{2}
\left\|
U|S|^{2(1-t)}U^{*} + |T|^{2}
\right\|
+ |\lambda|\,r_{\omega}(U|S|^{1-t}T)
\Bigg]\notag\\&=\|S\|^{t}
\Bigg[
\frac{\max\{|\lambda|,|1-\lambda|\}}{2}
\left\|
|S^*|^{2(1-t)}+ |T|^{2}
\right\|
+ |\lambda|\,r_{\omega}(U|S|^{1-t}T)
\Bigg]\tag{23}\end{align}

and 

\begin{align}\label{24}
r_{\omega}(ST)
\le
\|S\|^{t}
\Bigg[
|\lambda|\,r_{\omega}(U|S|^{1-t}T)
+\frac{\max\{|\lambda|,|1-\lambda|\}}{2}
\left(
\|U|S|^{1-t}\|\;\|T\|
+\|TU|S|^{1-t}\|
\right)
\Bigg] .\tag{24}
\end{align}
Using the fact $r_{\omega}(A^*B)\le\frac{1}{2}\|\,|A|^2+|B|^2\,\|$ with $A^{*} = U|S|^{1-t}$ and $B = T$, we get
\begin{align*}
r_{\omega}(U|S|^{1-t}T)
&\le
\frac{1}{2}
\left\|
U|S|^{2(1-t)}U^{*} + |T|^{2}
\right\|  \\
&=
\frac{1}{2}
\left\|
|S^*|^{2(1-t)} + |T|^{2}
\right\|.
\end{align*}
Now, from (\ref{23}) and (\ref{24}), we obtain the desired bounds in (a).

By setting $A = U|S|^{t}$, $X = |S|^{1-t}$ and $B = T$ in Theorem \ref{Theorem2}, and then, proceeding similar process as above, we get the required bounds in (b).

Let $T = V|T|$ be the polar decomposition of $T$.
By choosing $A = S$, $X = V|T|^{t}$ and $B = |T|^{1-t}$ in Theorem \ref{Theorem2}, we establish the inequalities in (c).

Finally, by considering $A = S$, $X = V|T|^{1-t}$ and $B = |T|^{t}$ in Theorem \ref{Theorem2}, we prove the required inequalities in (d).

\end{proof}

\begin{remark}
 (a)  By putting $\lambda=\frac12$ in Theorem \ref{Theorem2}, we recover the inequalities which were given in [\citenum{sababheh2025inner}, Theorem 2.4], namely
\begin{align*}
&r_{\omega}(AXB)
\le
\frac{\|X\|}{2}
\left(
r_{\omega}(AB)
+\frac12 \|AA^{*}+B^{*}B\|
\right) \qquad \text{and}
\\&r_{\omega}(AXB)
\le
\frac{\|X\|}{2}
\left(
r_{\omega}(AB)
+\frac12
\left(
\|A\|\,\|B\|+\|BA\|
\right)
\right).
\end{align*}

(b) By putting $\lambda=\frac12$ and $t=\frac12$ in Corollary \ref{Corollary9}, we get
\begin{align*}
&r_{\omega}(ST)
\le
\frac{\|S\|^{1/2}}{2}
\|\,|S^*|+|T|^2\,\| = \alpha\, \text{(say)}
,
\\&
r_{\omega}(ST)
\le
\frac{\|S\|^{1/2}}{2}
\left[
\frac12\|\,|S^*|+|T|^{2}\,\|
+\|S\|^{1/2}\;\|T\|
\right] = \beta \,\text{(say)},
\\&
r_{\omega}(ST)
\le
\frac{\|T\|^{1/2}}{2}
\|\,|S^*|^2+|T|\,\| = \gamma \,\text{(say)}\qquad\text{and}
\\&r_{\omega}(ST)
\le
\frac{\|T\|^{1/2}}{2}
\left[
\frac12\|\,|S^*|^2+|T|\,\|
+\|S\|\;\|T\|^{1/2}
\right] = \delta \,\text{(say)}.
\end{align*}
From the above inequalities, we can write 
\begin{align*}
&r_{\omega}(ST)\le \min(\alpha, \gamma)  \quad \text{and}
\\& r_{\omega}(ST)\le \min(\beta, \delta),
\end{align*} which refine the inequalities (\ref{7}) and (\ref{8}).

We address an example to show proper improvement:

Let
$S=\begin{bmatrix}
   2 & 1 & -1\\
   -1 & -3 & -5\\
   2 & -3 & 1
\end{bmatrix}$ and 
$T=\begin{bmatrix}
    1 & 2 & 3\\
    3 & -5 & 2\\
    -7 & 1 & 3
\end{bmatrix}.$ By using MATLAB, we get 
\begin{align*}
 &r_{\omega}(ST)<  \min(\alpha, \gamma)=\gamma= 58.815624< \alpha= 90.790303 \qquad \text{and}\\& r_{\omega}(ST)<\min(\beta, \delta)=\delta= 54.463033< \beta= 70.450373.
\end{align*}

\end{remark}

\section{Bounds for  \texorpdfstring{$\lambda-$} {lamda-} mean transform} \label{section3}

We first give an answer for Problem 1.
The inequality $r_{\sigma}(M_\lambda(T)) \le r_{\sigma}(T)$ is not universally valid for the spectral radius $r_{\sigma}(\cdot)$. Indeed, the spectral radius of the $\lambda$-mean transform $r_{\sigma}(M_\lambda(T))$ can be strictly greater than $r_{\sigma}(T)$ significantly.

For this, we consider the following example:
\begin{example}
Let $H = \ell^2(\mathbb{N})$ be the Hilbert space of square-summable complex sequences, with standard orthonormal basis $\{e_0, e_1, e_2, e_3, \dots\}$.

We construct $T$ as a unilateral weighted shift operator defined by:
\[
T e_n = \alpha_n e_{n+1} \quad \text{for all } n \ge 0,
\]
where the weights $\alpha = (\alpha_0, \alpha_1, \alpha_2, \dots)$ alternate between $1$ and a fixed parameter $\epsilon \in (0,1)$:
\[
\alpha = (1, \epsilon, 1, \epsilon, 1, \epsilon, \dots), \quad \text{specifically}
\]
\[
\alpha_n = \begin{cases} 
1 & \text{if } n \text{ is even} \\ 
\epsilon & \text{if } n \text{ is odd.} 
\end{cases}
\]

We get the adjoint operator $T^*$ of $T$, defined by:
\[
T^* e_0 = 0 \quad \text{and} \quad T^* e_n = \alpha_{n-1} e_{n-1} \quad \text{for } n \ge 1.
\]

Now,
\[T^* T e_n = T^*(\alpha_n e_{n+1}) = \alpha_n^2 e_n = |T|^2 e_n = \alpha_n^2 e_n \implies |T| e_n = \alpha_n e_n.\]

Because $\alpha_n > 0$ for all $n$, $\ker(|T|) = \{0\}$.
Since $T e_n = V |T| e_n = V(\alpha_n e_n) = \alpha_n V(e_n)$, and $T e_n = \alpha_n e_{n+1}$, dividing by $\alpha_n > 0$ yields:
\[
V e_n = e_{n+1} \quad \text{for all } n \ge 0.
\]

$V$ is simply unweighted unilateral shift isometry ($V^* V = I$, $\ker(V) = \{0\}$).

By Gelfand's Spectral Radius Formula, we have
\[r_{\sigma}(T) = \lim_{k \to \infty} \|T^k\|^{\frac{1}{k}}.\]
\[
\text{If  $n$ is even  $(n = 2j)$ then
$T^2 e_{2j} = T(1 \cdot e_{2j+1}) = 1 \cdot \epsilon \cdot e_{2j+2} = \epsilon \cdot e_{2j+2}.$}\]
\[
\text{If  $n$ is odd  $(n = 2j+1)$  then  
$T^2 e_{2j+1} = T(\epsilon \cdot e_{2j+2}) = 1 \cdot \epsilon \cdot e_{2j+3} = \epsilon \cdot e_{2j+3}.$}\]

Combining this, we get $T^2 e_n = \epsilon e_{n+2}$ for all $n$.
So, 

\[T^{2k} e_n = \epsilon^k e_{n+2k} \implies \|T^{2k}\| = \epsilon^k.\]

By applying Gelfand's formula, we infer that:
\[
r_{\sigma}(T) = \lim_{k \to \infty} \|T^{2k}\|^{\frac{1}{2k}} = \lim_{k \to \infty} (\epsilon^k)^{\frac{1}{2k}} = \sqrt{\epsilon}.
\]

Let us compute the Duggal transformation:
\[
T_D e_n = |T| V e_n = |T| e_{n+1} = \alpha_{n+1} e_{n+1}.
\]

Let us compute the $\lambda$-mean transformation:
\begin{align*}
M_\lambda(T) e_n &= \lambda T e_n + (1 - \lambda) T_D e_n \\
&= \lambda (\alpha_n e_{n+1}) + (1 - \lambda) (\alpha_{n+1} e_{n+1}) \\
&= (\lambda \alpha_n + (1 - \lambda) \alpha_{n+1}) e_{n+1} \\
&= \beta_n e_{n+1} \quad \text{where } \beta_n = \lambda \alpha_n + (1 - \lambda) \alpha_{n+1}.
\end{align*}

Now,

\[(M_\lambda(T))^2 e_n = M_\lambda(T)(\beta_n e_{n+1}) = \beta_n M_\lambda(T) e_{n+1} = \beta_n \beta_{n+1} e_{n+2},\] where 
each weight is $(\lambda + (1 - \lambda)\epsilon)(\lambda \epsilon + (1 - \lambda))= \epsilon + \lambda(1 - \lambda)(1 - \epsilon)^2.$

So,

\[\|(M_\lambda(T))^{2k}\| = (\epsilon + \lambda(1 - \lambda)(1 - \epsilon)^2)^k.\]

By applying Gelfand's formula, we reach:
\[
r_{\sigma}(M_\lambda(T)) = \lim_{k \to \infty} \left( (\epsilon + \lambda(1 - \lambda)(1 - \epsilon)^2)^k \right)^{\frac{1}{2k}} = \sqrt{\epsilon + \lambda(1 - \lambda)(1 - \epsilon)^2}.
\]

We have,
\[r_{\sigma}(M_\lambda(T))^2 - r_{\sigma}(T)^2 = \epsilon + \lambda(1 - \lambda)(1 - \epsilon)^2 - \epsilon = \lambda(1 - \lambda)(1 - \epsilon)^2.\]

Since $\lambda \in (0,1)$, $\lambda(1 - \lambda) > 0$ and $(1 - \epsilon)^2 > 0,$

 \[r_{\sigma}(M_\lambda(T))^2 - (r_{\sigma}(T))^2 > 0 \implies r_{\sigma}(M_\lambda(T)) > r_{\sigma}(T) \quad \text{for all } \lambda \in (0,1).\]

\end{example}

Before giving an answer for Problem 2, we define a new class of operators:

$T$ is said to be in the $\sigma$-class if $(U^*)^2|T| = |T|(U^*)^2,$ denoted by $T \in \sigma(\mathcal{H})$.

Indeed, 
  \[T \text{\,  is quasinormal} \implies U|T| = |T|U \implies (U|T|)^* = (|T|U)^* \implies |T|U^* = U^*|T|.\]
Now,
\[(U^*)^2|T| = U^*(U^*|T|) = U^*(|T|U^*) = |T|U^* U^* = |T|(U^*)^2.\]

i.e., $T \text{ is quasinormal} \implies T \text{ is } \sigma\text{-class}$.
But the converse is not true in general.

For example, we consider the Hilbert space $H = \mathbb{C}^2$. Let us consider
$T = \begin{bmatrix} 0 & 2 \\ 1 & 0 \end{bmatrix}.$ $\text{Then } |T| = \begin{bmatrix} 1 & 0 \\ 0 & 2 \end{bmatrix}.$
For $T = U|T|$, we get
\[
U = T|T|^{-1} = \begin{bmatrix} 0 & 2 \\ 1 & 0 \end{bmatrix} \begin{bmatrix} 1 & 0 \\ 0 & 1/2 \end{bmatrix} = \begin{bmatrix} 0 & 1 \\ 1 & 0 \end{bmatrix}\implies (U^{*})^2 = \begin{bmatrix} 0 & 1 \\ 1 & 0 \end{bmatrix} \begin{bmatrix} 0 & 1 \\ 1 & 0 \end{bmatrix} = \begin{bmatrix} 1 & 0 \\ 0 & 1 \end{bmatrix} = I.
\]
So, $(U^*)^2|T| = |T| = |T|(U^*)^2 \implies T$ is in $\sigma$-class.
Now,
\[U|T| = \begin{bmatrix} 0 & 1 \\ 1 & 0 \end{bmatrix} \begin{bmatrix} 1 & 0 \\ 0 & 2 \end{bmatrix} = \begin{bmatrix} 0 & 2 \\ 1 & 0 \end{bmatrix}
\quad \text{and} \quad
|T|U = \begin{bmatrix} 1 & 0 \\ 0 & 2 \end{bmatrix} \begin{bmatrix} 0 & 1 \\ 1 & 0 \end{bmatrix} = \begin{bmatrix} 0 & 1 \\ 2 & 0 \end{bmatrix}.
\]
Since $U|T| \neq |T|U$, $T$ is not quasi-normal.

The above example  demonstrates that quasi-normality is too strong an assumption. The $\sigma$-class ($(U^*)^2|T| = |T|(U^*)^2$) is the necessary and sufficient condition for this symmetry $(M_\lambda(T^*) = (M_\lambda(T))^*)$  to hold.

\begin{proposition} Let $T \in \mathcal{B}(\mathcal{H})$. Then $M_\lambda(T^*) = (M_\lambda(T))^*,$ for $\lambda\in [0,1)$ if and only if $T$ is in the $\sigma$-class, i.e., $(U^*)^2|T| = |T|(U^*)^2$.\end{proposition} 

\begin{proof}
Let us suppose that $T \in \sigma(\mathcal{H})$ with polar decomposition $T=U|T|=|T^*|U.$ That is \[T \in \sigma(\mathcal{H}) \implies (U^*)^2|T| = |T|(U^*)^2.\] Then, $T^D = |T|U$ is the Duggal transformation of $T$ and $T^* = U^*|T^*|$ is the polar decomposition of $T^*$.
So, the Duggal transformation of $T^*$ is $(T^*)^D = |T^*|U^*$.  

To show $M_\lambda(T^*) = (M_\lambda(T))^*,$ we need only to establish $(T^*)^D= (T^D)^*.$
Now,
\begin{align*}
(T^*)^D &= |T^*|U^* = (U|T|U^*) U^* \quad \left(\text{since } |T^*|^q = U|T|^q U^* \text{ for } q > 0\right) \\
&= U|T|(U^*)^2 = U(U^*)^2|T| = U^*|T| = (T^D)^*.
\end{align*}

Conversely, let us assume that $M_\lambda(T^*) = (M_\lambda(T))^*.$
\begin{align*}
&\implies \lambda T^* + (1-\lambda)(T^*)^D = \lambda T^* + (1-\lambda)(T^D)^* \\
&\implies (T^*)^D = (T^D)^*  \implies |T^*|U^* = U^*|T| \\
&\implies U^*(U^*|T|) = U^*|T^*|U^* = U^*(U|T|U^*)U^* \\
&\implies (U^*)^2|T| = |T|(U^*)^2 \implies T \text{ is in the } \sigma\text{-class.}
\end{align*}

\end{proof}
For Problem 3, we cannot expect any relation between $( r_{\omega}(M_\lambda(T)))^n $ and $r_{\omega}(T^n),$ where $\lambda \in [0, 1)$ and $n \ge 2.$

In the next two  examples, we shall show that there is no relationship between $( r_{\omega}(M_\lambda(T)))^n$ and $ r_{\omega}(T^n)$ for $n \ge 2$ and  $\lambda \in [0, 1)$. That means they are not comparable.
\begin{example}
Let $T = \begin{bmatrix} 0 & 1 & 0 \\ 0 & 0 & 1 \\ 0 & 0 & 0 \end{bmatrix}$ be the $3 \times 3$ nilpotent matrix. Then, we have 
\[ T^2 = \begin{bmatrix} 0 & 0 & 1 \\ 0 & 0 & 0 \\ 0 & 0 & 0 \end{bmatrix}, \,  T^*T =  \begin{bmatrix} 0 & 0 & 0 \\ 0 & 1 & 0 \\ 0 & 0 & 1 \end{bmatrix}\implies |T| = \begin{bmatrix} 0 & 0 & 0 \\ 0 & 1 & 0 \\ 0 & 0 & 1 \end{bmatrix}, \, U = \begin{bmatrix} 0 & 1 & 0 \\ 0 & 0 & 1 \\ 0 & 0 & 0 \end{bmatrix}, \, T_D  = \begin{bmatrix} 0 & 0 & 0 \\ 0 & 0 & 1 \\ 0 & 0 & 0 \end{bmatrix} \]\quad\text{and}
$M_\lambda(T) = \begin{bmatrix} 0 & \lambda & 0 \\ 0 & 0 & 1 \\ 0 & 0 & 0 \end{bmatrix}.$

 So, $ r_{\omega}(M_\lambda(T)) =  r_{\omega} \left(\begin{bmatrix} 0 & \lambda & 0 \\ 0 & 0 & 1 \\ 0 & 0 & 0 \end{bmatrix}\right) = \frac{\sqrt{1+\lambda^2}}{2}$.
Now, $T^2 = \begin{bmatrix} 0 & 0 & 1 \\ 0 & 0 & 0 \\ 0 & 0 & 0 \end{bmatrix}$, $ r_{\omega}(T^2) = \frac{1}{2}$.

It is clear that $( r_{\omega}(M_\lambda(T)))^2 = \frac{1+\lambda^2}{4} < \frac{1}{2}=r_{\omega}(T^2) $ for $\lambda \in [0, 1)$.

\end{example}

\begin{example}
Let $T = \begin{bmatrix} 0 & 1 \\ 2 & 0 \end{bmatrix}$. Then, we have

\[T^2 = \begin{bmatrix} 2 & 0 \\ 0 & 2 \end{bmatrix}, \,
|T| = \begin{bmatrix} 2 & 0 \\ 0 & 1 \end{bmatrix}, \, U = \begin{bmatrix} 0 & 1 \\ 1 & 0 \end{bmatrix}, \, T^D = |T|U = \begin{bmatrix} 2 & 0 \\ 0 & 1 \end{bmatrix}\begin{bmatrix} 0 & 1 \\ 1 & 0 \end{bmatrix} = \begin{bmatrix} 0 & 2 \\ 1 & 0 \end{bmatrix}.\] 

We obtain 

\[ r_{\omega}(M_{\lambda}(T)) = r_{\omega}\left(\begin{bmatrix}
    0 & 2-\lambda\\
    1+\lambda & 0
\end{bmatrix}\right)=\frac{3}{2}, \,  r_{\omega}(T^2) = 2.\]

That gives   $( r_{\omega}(M_{\lambda}(T)))^2 = \frac{9}{4} > 2= r_{\omega}(T^2) $.

\end{example}

Recall the tensor product $H_1\otimes H_2$ of Hilbert spaces $H_1$ and $H_2$ is the completion of the inner product space of elements of the form $\sum_{i=1}^{n} x_i \otimes y_i$, with $x_i \in H$ and $y_i \in K$, and $n \ge 1$, under the inner product $\langle x \otimes y, u \otimes v \rangle = \langle x, u \rangle \langle y, v \rangle$. 
The tensor product of two operators $A \in B(H_1)$ and $B \in B(H_2)$ is denoted by $A \otimes B$ and is defined on $H_1 \otimes H_2$ by $(A \otimes B)(x \otimes y) = Ax \otimes By$ for $x \otimes y \in H_1 \otimes H_2$.
We can write $A \otimes B$ as the product of $A \otimes I_{H_2}$ and $I_{H_1} \otimes B$, i.e., $A \otimes B = (A \otimes I_{H_2})(I_{H_1} \otimes B)$, where $I_{H_1}$ and $I_{H_2}$ are the identity operators on $H_1$ and $H_2$, respectively. 

It is well-known that $S \otimes T$ is a bounded linear operator acting on $H_1 \otimes H_2$. Furthermore, it satisfies the following properties:
\begin{proposition}
     Let $S, S_1, S_2 \in \mathcal{B}(H_1)$ and $T, T_1, T_2 \in \mathcal{B}(H_2)$. Then:
\begin{enumerate}
    \item $\|S \otimes T\| = \|S\| \|T\|$.
    \item $\gamma \delta (S \otimes T) = \gamma S \otimes \delta T$, for all $\gamma, \delta \in \mathbb{C}$.
    \item $(S_1 + S_2) \otimes (T_1 + T_2) = S_1 \otimes T_1 + S_1 \otimes T_2 + S_2 \otimes T_1 + S_2 \otimes T_2$.
    \item $(S_1 \otimes T_1)(S_2 \otimes T_2) = S_1 S_2 \otimes T_1 T_2$. In particular, we have:
    
   $ S \otimes T = (S \otimes I_K)(I_H \otimes T  ) = (I_H \otimes T)(S \otimes I_K).$
\item $(S \otimes T)^* = S^* \otimes T^*$.
    \end{enumerate}
\end{proposition}
Suppose $S \in \mathcal{B}(\mathcal{H}_1)$ and $T \in \mathcal{B}(\mathcal{H}_2)$ with polar decompositions $S= U|S|,$ $T = V|T|.$ The polar decomposition and the Duggal transform of $S \otimes T$ is, respectively

$S \otimes T =U|S| \otimes V|T|  = (U \otimes V) (|S| \otimes |T|) =(U \otimes V) |S \otimes T|,$
and $(S \otimes T)^D = |S \otimes T| (U \otimes V)= (|S| \otimes |T|) (U \otimes V) = |S|U \otimes |T|V = S^D \otimes T^D.$

So, the $\lambda$-mean transform of $S \otimes T$ is
$M_\lambda(S \otimes T) = \lambda (S \otimes T) + (1 - \lambda) (S \otimes T)^D = \lambda (S \otimes T) + (1 - \lambda) (S^D \otimes T^D).$
Let $S \in \mathcal{B}(\mathcal{H}_1)$ and $T \in \mathcal{B}(\mathcal{H}_2)$. Then,
$M_{\lambda}(S \otimes T)$ acts on the tensor product Hilbert space $\mathcal{H}_1 \otimes \mathcal{H}_2$ and
$M_{\lambda}(T \otimes S)$ acts on the tensor product Hilbert space $\mathcal{H}_2 \otimes \mathcal{H}_1.$
 Since tensor product multiplication is non-commutative ($A \otimes B \neq B \otimes A$ in general), so $M_{\lambda}(S \otimes T) \neq M_{\lambda}(T \otimes S)$ for $\lambda \in (0,1)$.
 We consider the following example:
\[
S = \begin{pmatrix} 1 & 0 \\ 0 & 0 \end{pmatrix} \quad \text{and} \quad T = \begin{pmatrix} 0 & 1 \\ 0 & 0 \end{pmatrix}.
\]
$S$ is positive semi-definite, so $U = S$, $|S| = S$.
$T$ has canonical polar decomposition $T = V |T|$ with $V = \begin{pmatrix} 0 & 1 \\ 0 & 0 \end{pmatrix}$ and $|T| = \begin{pmatrix} 0 & 0 \\ 0 & 1 \end{pmatrix}$. Now,
\begin{align*}
M_{\frac{1}{2}}(S \otimes T) = \frac{1}{2} \begin{pmatrix} 0 & 1 & 0 & 0 \\ 0 & 0 & 0 & 0 \\ 0 & 0 & 0 & 0 \\ 0 & 0 & 0 & 0 \end{pmatrix}\quad \text{and} \quad
M_{\frac{1}{2}}(T \otimes S) = \frac{1}{2} \begin{pmatrix} 0 & 0 & 1 & 0 \\ 0 & 0 & 0 & 0 \\ 0 & 0 & 0 & 0 \\ 0 & 0 & 0 & 0 \end{pmatrix}.
\end{align*}
 It is clear that $M_{\frac{1}{2}}(S \otimes T) \neq M_{\frac{1}{2}}(T \otimes S)$.

Let $T$ and $S$ be bounded linear operators on complex Hilbert spaces $\mathcal{H}_1$ and $\mathcal{H}_2$, respectively. Then $S$ is said to dilate $T$ if $T=U^*SU$ for some isometry $U:\mathcal{H}_1\to\mathcal{H}_2$.

\begin{proposition} \label{proposition 8}
  Let $S \in \mathcal{B}(\mathcal{H}_1)$ and $T \in \mathcal{B}(\mathcal{H}_2).$  Then $M_{\lambda}(T \otimes S)$ and $M_{\lambda}(S \otimes T)$ are unitarily equivalent dilations of each other.
\end{proposition}

\begin{proof} Let us define the canonical flip operator $V : H_1 \otimes H_2 \to H_2 \otimes H_1$ by $V(x \otimes y) = y \otimes x$.
Now, 
$V(V(x \otimes y)) = V(y \otimes x) = x \otimes y \implies V^2 = I.$

In a tensor product Hilbert space, the inner product of two elementary tensors is the product of their individual inner products, i.e.
$\langle x_1 \otimes x_2, y_1 \otimes y_2 \rangle = \langle x_1, y_1 \rangle \langle x_2, y_2 \rangle.$

Apply the flip operator:
\begin{align*}
\langle V(x_1 \otimes x_2), V(y_1 \otimes y_2) \rangle = \langle x_2 \otimes x_1, y_2 \otimes y_1 \rangle 
= \langle x_2, y_2 \rangle \langle x_1, y_1 \rangle 
\end{align*}

Because these are just complex scalars, their multiplication commutes: 
\begin{align*}
\langle V^* V (x_1 \otimes x_2), y_1 \otimes y_2 \rangle&=\langle x_1, y_1 \rangle \langle x_2, y_2 \rangle  = \langle x_1 \otimes x_2, y_1 \otimes y_2 \rangle\\&\implies V^* V = I.\end{align*}

Since $V^* V = I$ and $V^2 = I \implies V^* = V^{-1} = V$.
Now,
\begin{align*}
V M_\lambda(S \otimes T) V^* (x \otimes y) &= V M_\lambda(S \otimes T) (y \otimes x) \\
&= V \left( \lambda (S \otimes T)(y \otimes x) + (1-\lambda) (S^D \otimes T^D)(y \otimes x) \right) \\
&= \lambda V(Sy \otimes Tx) + (1-\lambda) V(S^D y \otimes T^D x) \\
&= \lambda (Tx \otimes Sy) + (1-\lambda) (T^D x \otimes S^D y) \\
&= \lambda (T \otimes S)(x \otimes y) + (1-\lambda) (T^D \otimes S^D)(x \otimes y) \\
&= \left( \lambda (T \otimes S) + (1-\lambda) (T^D \otimes S^D) \right)(x \otimes y)= M_\lambda(T \otimes S)(x \otimes y)
\end{align*}
That gives 
 $V M_\lambda(S \otimes T) V^* = M_\lambda(T \otimes S).$
 
By multiplying $V^*$ on the left and $V$ on the right, we get:
\begin{align*}
M_{\lambda}(S \otimes T) = V^* M_{\lambda}(T \otimes S) V.\end{align*}
\end{proof}

We shall commence with the subsequent properties of the $\lambda-$mean transformation of $T$ and $T \otimes S.$
\begin{lemma}\label{Lemma11}
 Let $S \in \mathcal{B}(\mathcal{H}_1)$ and $T \in \mathcal{B}(\mathcal{H}_2).$ Then the following properties hold: 
 \begin{enumerate}
     \item $M_{\lambda}(cS)=cM_{\lambda}(S)$ for every complex number $c$.
     \item $M_{\lambda}(USU^*)=M_{\lambda}(S)$ for every unitary operator $U$.
     \item $r_{\sigma}(M_{\lambda}(T \otimes S))=r_{\sigma}(M_{\lambda}(S \otimes T)).$
     \item $r_{\omega}(M_{\lambda}(T \otimes S))=r_{\omega}(M_{\lambda}(S \otimes T)).$
     \item $\|M_{\lambda}(T \otimes S))\|=\|M_{\lambda}(S \otimes T))\|.$
 \end{enumerate}
\end{lemma}

\begin{proof}
  (1) Let $S= U\vert{}S\vert{}$ be the polar decomposition of $S \in \mathcal{B}(\mathcal{H}).$
  Let $c \in \mathbb{C} \setminus \{0\}$ be in polar form $c = \vert{}c\vert{}e^{i\theta}.$
  Then $\vert{}cS\vert{} = \left((cS)^*(cS)\right)^{1/2} = \left(\bar{c}c \, S^*S\right)^{1/2} = \sqrt{\vert{}c\vert{}^2} (S^*S)^{1/2} = \vert{}c\vert{}\vert{}S\vert{}$ and 
  $cS = (e^{i\theta}\vert{}c\vert{})(U\vert{}S\vert{}) = (e^{i\theta}U)(\vert{}c\vert{}\vert{}S\vert{}).$
  
  Substituting $V = e^{i\theta}U$ and $\vert{}cS\vert{} = \vert{}c\vert{}\vert{}S\vert{}$ directly into the transform definition:
  \begin{align*}
  M_\lambda(cS) &= \lambda V \vert{}cS\vert{} + (1-\lambda) \vert{}cS\vert{} V= \lambda (e^{i\theta}U)(\vert{}c\vert{}\vert{}S\vert{}) + (1-\lambda)(\vert{}c\vert{}\vert{}S\vert{})(e^{i\theta}U) \\&= \vert{}c\vert{}e^{i\theta} \Big( \lambda U\vert{}S\vert{} + (1-\lambda)\vert{}S\vert{}U \Big)= c M_\lambda(S).\end{align*}

(2) Let $S = V\vert{}S\vert{}$ be the polar decomposition of $S\in \mathcal{B}(\mathcal{H})$ and $T' = U S U^*.$ Then, 
\begin{align*}
|T'|^2 = (T')^* T' = U S^* U^* U S U^*& = U |S|^2 U^* = (U |S| U^*) (U |S| U^*) = (U |S| U^*)^2\\&
\implies |T'| = U |S| U^*.
\end{align*}

Now,  
\[T' = U S U^* = U V |S| U^* = (U V U^*) U |S| U^*= (U V U^*) |T'| = U_{1} |T'|,\]
where $U_{1} = U V U^*$ is a partial isometry.
We have,
\[(T')^D = |T'| U' = U |S| U^* U V U^* = U S^D U^*.\]  So, 
\begin{align*}M_\lambda (U S U^*) &= \lambda T' + (1-\lambda) (T')^D = \lambda U S U^* + (1-\lambda) U S^D U^* \\&= U \left( \lambda S + (1-\lambda) S^D \right) U^* = U M_\lambda(S) U^*.\end{align*}

(3)  By using Proposition \ref{proposition 8}, we have
    \begin{align*}
    &M_\lambda(T \otimes S) - zI = V M_\lambda(S \otimes T) V^* - z V V^* = V \left( M_\lambda(S \otimes T) - zI \right) V^*\\
    &\implies (M_\lambda(T \otimes S) - zI)^{-1} = V \left( M_\lambda(S \otimes T) - zI \right)^{-1} V^{*}.
    \end{align*}
    Thus, $(M_\lambda(T \otimes S) - zI)$ is invertible if and only if $(M_\lambda(S \otimes T) - zI)$ is invertible. This gives
    \[
    r_{\sigma}(M_\lambda(T \otimes S)) = r_{\sigma}(M_\lambda(S \otimes T)).
    \]

(4) If $\|x\| = 1$ and $V$ is unitary, then $\|Vx\| = \|V^* x\| = 1$. By applying Proposition \ref{proposition 8}, we have \quad$r_{\omega}(M_\lambda(T \otimes S)) = \sup\limits_{\|x\|=1} |\langle M_\lambda(T \otimes S)x, x \rangle|$
  \begin{align*}  
    &\qquad\qquad\qquad\qquad= \sup\limits_{\|x\|=1} |\langle V M_\lambda(S \otimes T) V^* x, x \rangle|
    = \sup\limits_{\|x\|=1} |\langle M_\lambda(S \otimes T) V^* x, V^* x \rangle| 
   \\& \qquad\qquad\qquad\qquad= \sup\limits_{\|y\|=1} |\langle M_\lambda(S \otimes T) y, y \rangle| 
    = r_{\omega}(M_\lambda(S \otimes T)).\end{align*}.

(5)  By applying Proposition \ref{proposition 8}, we have 
    \begin{align*} \|M_\lambda(T \otimes S)\|&= \sup\limits_{\substack{x, y \in {\mathcal{H}}\\ \|x\|=\|y\|=1}} |\langle V M_\lambda(S \otimes T) V^* x, y \rangle|  = \sup\limits_{\substack{x, y \in {\mathcal{H}} \\ \|x\|=\|y\|=1}} |\langle M_\lambda(S \otimes T) V^* x, V^* y \rangle|\\& = \sup\limits_{\substack{z, w \in {\mathcal{H}} \\ \|z\|=\|w\|=1}} |\langle M_\lambda(S \otimes T) z, w \rangle|= \|M_\lambda(S \otimes T)\|.\end{align*}
    
\end{proof}

\begin{theorem} The generalized Aluthge transform satisfies 
\[
\Delta_\lambda(S \otimes T) = \Delta_\lambda(S) \otimes \Delta_\lambda(T)
\]
for all bounded linear operators $S$ and $T$ and for all $\lambda \in [0, 1]$.
\end{theorem}
\begin{proof} If $S = U|S|$ and $T = V|T|$ are the polar decompositions of $S$ and $T$, respectively, then
$S \otimes T=U|S| \otimes V|T|= (U \otimes V) (|S| \otimes |T|).$

The partial isometry of $S \otimes T$ is $U_{S \otimes T} = U \otimes V$.
The modulus of $S \otimes T$ is $|S \otimes T| = |S| \otimes |T|$.

For any continuous function $f$ (and specifically for power functions $x \to x^\alpha$), we have 
$|S \otimes T|^\alpha = |S|^\alpha \otimes |T|^\alpha.$

Substituting these components into the definition of $\Delta_\lambda(S \otimes T)$, we get

\begin{align*}\Delta_\lambda(S \otimes T) &= |S \otimes T|^\lambda U_{S \otimes T} |S \otimes T|^{1-\lambda} = \left(|S|^\lambda \otimes |T|^\lambda\right) (U \otimes V) \left(|S|^{1-\lambda} \otimes |T|^{1-\lambda}\right) \\&= \left(|S|^\lambda U |S|^{1-\lambda}\right) \otimes \left(|T|^\lambda V |T|^{1-\lambda}\right) = \Delta_\lambda(S) \otimes \Delta_\lambda(T).\end{align*}

\end{proof}

An operator $T \in B(H)$ is classified as normal if it satisfies the condition $TT^* = T^*T$, and it is deemed quasinormal if it commutes with $T^*T.$ It is evident that $\text{normal} \Rightarrow \text{quasinormal}$, but the converse is not valid. It is widely recognised that an operator $T$ is quasinormal if and only if $T=T^D$.

It is clear that $M_{\lambda}(S \otimes T) =S \otimes T= M_{\lambda}(S) \otimes M_{\lambda}(T)$ when $\lambda=1.$ Also, for $\lambda=0,$ we have $M_{\lambda}(S \otimes T) =S^{D} \otimes T^{D}= M_{\lambda}(S) \otimes M_{\lambda}(T).$ 

Furthermore, for any $\lambda \in (0,1)$, $M_{\lambda}(S \otimes T) \neq M_{\lambda}(S) \otimes M_{\lambda}(T),$ in general.

\begin{theorem} Let $S, T \in \mathcal{B}(\mathcal{H})$ and let $\lambda \in (0,1)$. Then,$$M_{\lambda}(S \otimes T) = M_{\lambda}(S) \otimes M_{\lambda}(T)$$if and only if at least one of the following holds:

\begin{enumerate}
    \item[(i)] $S$ is quasinormal; or
    \item[(ii)] $T$ is quasinormal.
\end{enumerate}

\end{theorem}

\begin{proof} Let $S = U|S|$ and $T = V|T|$ be the polar decomposition of $S$ and $T$, respectively.

The canonical polar decomposition of the tensor product $S \otimes T$ is given by $(U \otimes V)(|S| \otimes |T|)$.

Applying the definition of the $\lambda$-mean transform $M_{\lambda}(A) = \lambda U |A| + (1-\lambda) |A| U$, we have
\begin{align*}
M_{\lambda}(S \otimes T) = \lambda (U \otimes V)(|S| \otimes |T|) + (1-\lambda) (|S| \otimes |T|)(U \otimes V).\end{align*}

Using the property $(A \otimes B)(C \otimes D) = AC \otimes BD$, we get
\begin{align*}
M_{\lambda}(S \otimes T) = \lambda (U|S| \otimes V|T|) + (1-\lambda) (|S|U \otimes |T|V).\end{align*}

Expanding the tensor product of the individual transforms, we infer that
\begin{align*}
M_{\lambda}(S) \otimes M_{\lambda}(T)& = \big(\lambda U|S| + (1-\lambda)|S|U\big) \otimes \big(\lambda V|T| + (1-\lambda)|T|V\big) \\
&= \lambda^2 (U|S| \otimes V|T|) + \lambda(1-\lambda)(U|S| \otimes |T|V)  + \\&\lambda(1-\lambda)(|S|U \otimes V|T|) + (1-\lambda)^2 (|S|U \otimes |T|V).
\end{align*}

Subtracting $M_{\lambda}(S) \otimes M_{\lambda}(T)$ from $M_{\lambda}(S \otimes T)$, we get
\begin{align*}
&M_{\lambda}(S \otimes T) - M_{\lambda}(S) \otimes M_{\lambda}(T) 
\\&= \lambda(1-\lambda) \big[ (U|S| \otimes V|T|) + (|S|U \otimes |T|V) - (U|S| \otimes |T|V) - (|S|U \otimes V|T|) \big] \\
&= \lambda(1-\lambda) (U|S| - |S|U) \otimes (V|T| - |T|V).
\end{align*}

For $\lambda \in (0,1)$, $\lambda(1-\lambda) \neq 0$, and $X \otimes Y = 0$ if and only if $X = 0$ or $Y = 0$.

Therefore, the equality holds if and only if
\[
U|S| - |S|U = 0 \quad \text{or} \quad V|T| - |T|V = 0.
\]

Equivalently, $U|S| = |S|U$ or $V|T| = |T|V$.
This condition is satisfied whenever $S$ or $T$ is quasinormal.

\end{proof}

\begin{remark}
$M_\lambda(S \otimes T)  = M_\lambda(T) \otimes M_\lambda(S)$ is false in general, even when both $S$ and $T$ are quasinormal. Indeed, $S \otimes T \neq T \otimes S,$ in general, and the quasinormality of $S$ and $T$ gives $M_\lambda(S \otimes T)  =S \otimes T$ and  $ M_\lambda(T) \otimes M_\lambda(S)= T \otimes S.$ 
\end{remark}

It is well known that all self-adjoint operators are normal, whereas the reverse is not always applicable. The subsequent theorem illustrates that the condition $M_\lambda(T) = M_\lambda(T^*)$ acts as a precise requirement for a normal operator $T$ to qualify as self-adjoint.
\begin{theorem}Let $T \in \mathcal{B}(\mathcal{H})$. Then the following statements are equivalent:
\begin{enumerate}
    \item[(a)] $T$ is normal and $M_\lambda(T) = M_\lambda(T^*)$.
    \item[(b)] $T$ is self-adjoint.
\end{enumerate}
\end{theorem}
\begin{proof}
$(b) \implies (a):$ It is obvious.

$(a) \implies (b):$  We know $\text{normal} \subset \text{quasinormal}$.
 $T$ is normal $\implies T, T^*$ both are normal
\begin{align*}
&\implies T, T^* \text{ both are quasinormal} \\
&\implies M_\lambda(T) = T \text{ and } M_\lambda(T^*) = T^*.
\end{align*} Now,
\begin{align*}
  &M_\lambda(T) = M_\lambda(T^*) \implies T = T^* \implies T \text{ is self-adjoint.}
\end{align*}

\end{proof}

A helpful bound for sums of double operator products in $\mathcal{B}(\mathcal{H}))$ is given by the following theorem.
\begin{theorem}\label{Theorem3}

Let $A_i,B_i,X_i \in {\mathcal{B}(\mathcal{H})}$ for $i=1,2,\cdots,n.$    Then the following inequality holds:
\begin{align}\label{25}
\left\|\sum\limits_{i=1}^{n}A_iX_iB_i\right\|
\le
\frac{\max\limits_{1\le i \le n}\|X_i\|}{2} \left(\left\|\sum\limits_{i=1}^{n}A_iB_i\right\|+ \left\|\sum\limits_{i=1}^nA_iA^*_{i}\right\|^{\frac{1}{2}}\;\left\|\sum\limits_{i=1}^nB^*_{i}B_{i}\right\|^{\frac{1}{2}}\right).
\tag{25}
\end{align}
\end{theorem}
\begin{proof}:
Suppose that  $\widehat{A}=
\begin{bmatrix}
A_1 & A_2 & \cdots & A_n\\
0 & 0 & \cdots & 0\\
\vdots & \vdots & \ddots & \vdots\\
0 & 0 & \cdots & 0
\end{bmatrix}_{n\times n},
\qquad
\widehat{X}=
\begin{bmatrix}
X_1 & 0 & \cdots & 0\\
0 & X_2 & \cdots & 0\\
\vdots & \vdots & \ddots & \vdots\\
0 & 0 & \cdots & X_n
\end{bmatrix}_{n\times n}$

and 
$\widehat{B}=
\begin{bmatrix}
0 & 0 & \cdots & B_1\\
0 & 0 & \cdots & B_2\\
\vdots & \vdots & \ddots & \vdots\\
0 & 0 & \cdots & B_n
\end{bmatrix}_{n\times n}.$

By simple calculation, we obtain
$\widehat{A}\widehat{X}\widehat{B}
=
\begin{bmatrix}
0 & 0 & \cdots & \sum\limits_{i=1}^{n}A_iX_iB_i\\
0 & 0 & \cdots & 0\\
\vdots & \vdots & \ddots & \vdots\\
0 & 0 & \cdots & 0
\end{bmatrix}_{n \times n}.$

Utilizing the facts that for any  $T \in \mathcal{B}(\mathcal{H}),$ $r_{\omega}(T)=\frac{1}{2}\|T\|$ if $T^2=0$, and $\|T\|^2=\|TT^*\|=\|T^*T\|,$ and for any  $Z_i \in \mathcal{B}(\mathcal{H}),$ where $i=1,2,\cdots,n,$  $\left\|
\begin{bmatrix}
0 & \cdots & 0 & Z_1 \\
\vdots & \ddots & \iddots & \vdots \\
0 & Z_{n-1} & \cdots & 0 \\
Z_n & 0 & \cdots & 0
\end{bmatrix}
\right\|=\left\|
\begin{bmatrix}
Z_1 & 0 & \cdots & 0\\
0 & Z_2 & \cdots & 0\\
\vdots & \vdots & \ddots & \vdots\\
0 & 0 & \cdots & Z_n
\end{bmatrix}
\right\|=\max\limits_{1\le i \le n}\|Z_i\|,$  we infer that
\begin{align*}
r_{\omega}(\widehat{A}\widehat{X}\widehat{B})=\tfrac12 \left\|\begin{bmatrix}
0 & 0 & \cdots & \sum\limits_{i=1}^{n}A_iX_iB_i\\
\vdots & \vdots & \ddots & \vdots\\
0 & 0 & \cdots & 0
\end{bmatrix}\right\|
&=\tfrac12 \left\|\sum\limits_{i=1}^{n}A_iX_iB_i\right\|.
\end{align*}

By applying the inequality (\ref{22}) with $\lambda=\frac12$, we get
\begin{align*}
r_{\omega}(\widehat{A}\widehat{X}\widehat{B})
&\le \frac{\|\widehat{X}\|}{2}\Big(
r_{\omega}(\widehat{A}\widehat{B})
+\tfrac12
\|\widehat{A}\|\;\|\widehat{B}\|+\frac{1}{2} \|\widehat{B}\widehat{A}\|
\Big)\\&
=\frac{\max\limits_{1\le i \le n}\|X_i\|}{2}
\Biggl[r_{\omega}\left(\begin{bmatrix}
0 & 0 & \cdots & \sum\limits_{i=1}^{n}A_iB_i\\
\vdots & \vdots & \ddots & \vdots\\
0 & 0 & \cdots & 0
\end{bmatrix}\right)+\\&\frac{1}{2}\left\|\begin{bmatrix}  \sum\limits_{i=1}^nA_iA^*_{i}&0&\cdots& 0\\ \vdots & \vdots  & \ddots & \vdots\\0 & 0  & \cdots & 0 \end{bmatrix}\right\|^{\frac{1}{2}}\;\left\|\begin{bmatrix}  0&0&\cdots& 0\\ \vdots & \vdots  & \ddots & \vdots\\0 & 0  & \cdots & \sum\limits_{i=1}^nB^*_{i}B_{i}\end{bmatrix}\right\|^{\frac{1}{2}}
\Biggr]\\&=\frac{\max\limits_{1\le i \le n}\|X_i\|}{4} \left(\left\|\sum\limits_{i=1}^{n}A_iB_i\right\|+ \left\|\sum\limits_{i=1}^nA_iA^*_{i}\right\|^{\frac{1}{2}}\;\left\|\sum\limits_{i=1}^nB^*_{i}B_{i}\right\|^{\frac{1}{2}}\right).
\end{align*}

This completes the proof.

\end{proof}

With this Theorem \ref{Theorem3}, we now enhance the inequality  (\ref{10}) proposed by Zamani:

\begin{theorem}\label{Theorem4}
Let $T = U|T|$ be the polar decomposition of $T \in \mathcal{B}(\mathcal{H})$.
Then
\begin{equation}\label{26}
\|M_{\lambda}(T)\|
\le
\max\{\|T\|^{t}, \|T\|^{1-t}\}\, \|\lambda |T^*|^{t} + (1-\lambda)|T|^{1-t}\|\tag{26}
\end{equation}
and
\begin{equation}\label{27}
\|M_{\lambda}(T)\|
\le
\sqrt{\|\lambda |T^*|^{2t} + (1-\lambda)|T|^{2(1-t)}\|\,\left(\lambda \|T\|^{2(1-t)} + (1-\lambda)\|T\|^{2t}\right)}, \quad \lambda,t\in[0,1].\tag{27}
\end{equation}
\end{theorem}

\begin{proof}: For $n=2,$  the inequality (\ref{25}) becomes 

\begin{align*}
\|A_1X_1B_1 + A_2X_2B_2\|
&\le\frac{\max\{\|X_1\|,\|X_2\|\}}{2}\Big[\|A_1B_1 + A_2B_2 \|\\&
+ \|A_1A_1^{*}+A_2A_2^{*}\|^{1/2}\;\|B_1^{*}B_1 + B_2^{*}B_2\|^{1/2}\Big].\end{align*}

Letting 
$A_1=\sqrt{\lambda} U|T|^{\frac{t}{2}}, 
\, 
X_1=|T|^{1-t}, 
\, 
B_1=\sqrt{\lambda}|T|^{\frac{t}{2}}, \,
A_2=\sqrt{1-\lambda}\,|T|^{\frac{1-t}{2}}, 
\, 
X_2=|T|^{t}, \,
\text{and}\, 
B_2=\sqrt{1-\lambda}\,|T|^{\frac{1-t}{2}}U
$
in the above inequality, we get
\begin{align*}
\|M_{\lambda}(T)\|
&\le
\frac12
\max\{\|T\|^{1-t},\|T\|^{t}\}
\Big[
\|\lambda U|T|^{t}+(1-\lambda)|T|^{1-t}U\|  \\
&\qquad
+
\|\lambda U|T|^{t}U^{*}
+(1-\lambda)|T|^{1-t}\|^{1/2}\,
\|
\lambda |T|^{t}
+(1-\lambda)U^*|T|^{1-t}U
\|^{1/2}
\Big].
\end{align*}

Using the fact that for any $A \in \mathcal{B}(\mathcal{H})$ 
and unitary operator $U,$

$U|A|^{p}=|A^*|^{p}U,\,U|A|^{p}U^{*}=|A^*|^{p}, 
\,U^{*}|A^*|^{p}U=|A|^{p},
\, (p>0),$
we obtain
\begin{align*}
\|M_{\lambda}(T)\|
&\le
\frac12
\max\{\|T\|^{1-t},\|T\|^{t}\}
\Big[
\|(\lambda |T^*|^{t}+(1-\lambda)|T|^{1-t})U\|  \\
&\quad
+
\|U^*(\lambda U |T|^{t}U^*+(1-\lambda)|T|^{1-t})U\|^{1/2}\;
\|
\lambda |T^*|^{t}+(1-\lambda)|T|^{1-t}
\|^{1/2}
\Big].
\end{align*}
This provides the inequality (\ref{26}).

Again, by letting
$X=Y=I,$ $A_1=\sqrt{\lambda}\,U|T|^{t},$                            $ A_2=\sqrt{\lambda}\,|T|^{1-t},$ $B_1=\sqrt{1-\lambda}\,|T|^{1-t},$
$B_2=\sqrt{1-\lambda}\,|T|^{t}U$ in (\ref{25}), we obtain the required inequality (\ref{27}).
\end{proof}

\begin{remark}
 We now demonstrate that the bound in Theorem \ref{Theorem4} (for $t=\frac{1}{2}$) is not greater than the upper bound in (\ref{10}).   
By using the fact  $\|A^r\|\le\|A\|^r,$ for any positive bounded operator $A$ and $r\in(0,1],$ and by the triangle inequality,   we see
\begin{align*}
\|M_{\lambda}(T)\|
&\le \|T\|^{\frac12}\,\|\lambda |T^*|^{\frac{1}{2}} + (1-\lambda)|T|^{\frac{1}{2}}\| \\&\le \|T\|^{\frac12}
    \left(\lambda\||T|\|^{\frac12} + (1-\lambda)\||T|\|^{\frac12} \right)= \|T\| \text{\qquad{and}}\\
\|M_{\lambda}(T)\|
&\le \|\lambda |T^*| + (1-\lambda)|T|\|^{\frac12}
     \left(\lambda \|T\| + (1-\lambda)\|T\|\right)^{\frac12} \\
&\le \left(\lambda\|T\| + (1-\lambda)\|T\|\right)^{\frac12}
     \|T\|^{\frac12} = \|T\|.
\end{align*}

For comparison, we consider a $3\times 3$ matrix 
$A=\begin{bmatrix}
 5-3i&3+i&-2-2i\\
 -2+2i&5+2i&2-i\\
 4+i&-5-5i&-2-3i
\end{bmatrix}.$
Let $g_1(\lambda)=\|A\|^{\frac{1}{2}}\|\lambda |A^*|^{\frac{1}{2}} + (1-\lambda)|A|^{\frac{1}{2}}\|,$ $g_2(\lambda)=\| \lambda |A^*| + (1-\lambda)|A|\|^{\frac12}
     \|A\|^{\frac12}$ and $\|M_{\lambda}(A)\|=g_3(\lambda).$
\begin{figure}[ht]
    \centering
     \begin{subfigure}{0.41\textwidth}
        \centering
        \includegraphics[width=\linewidth]{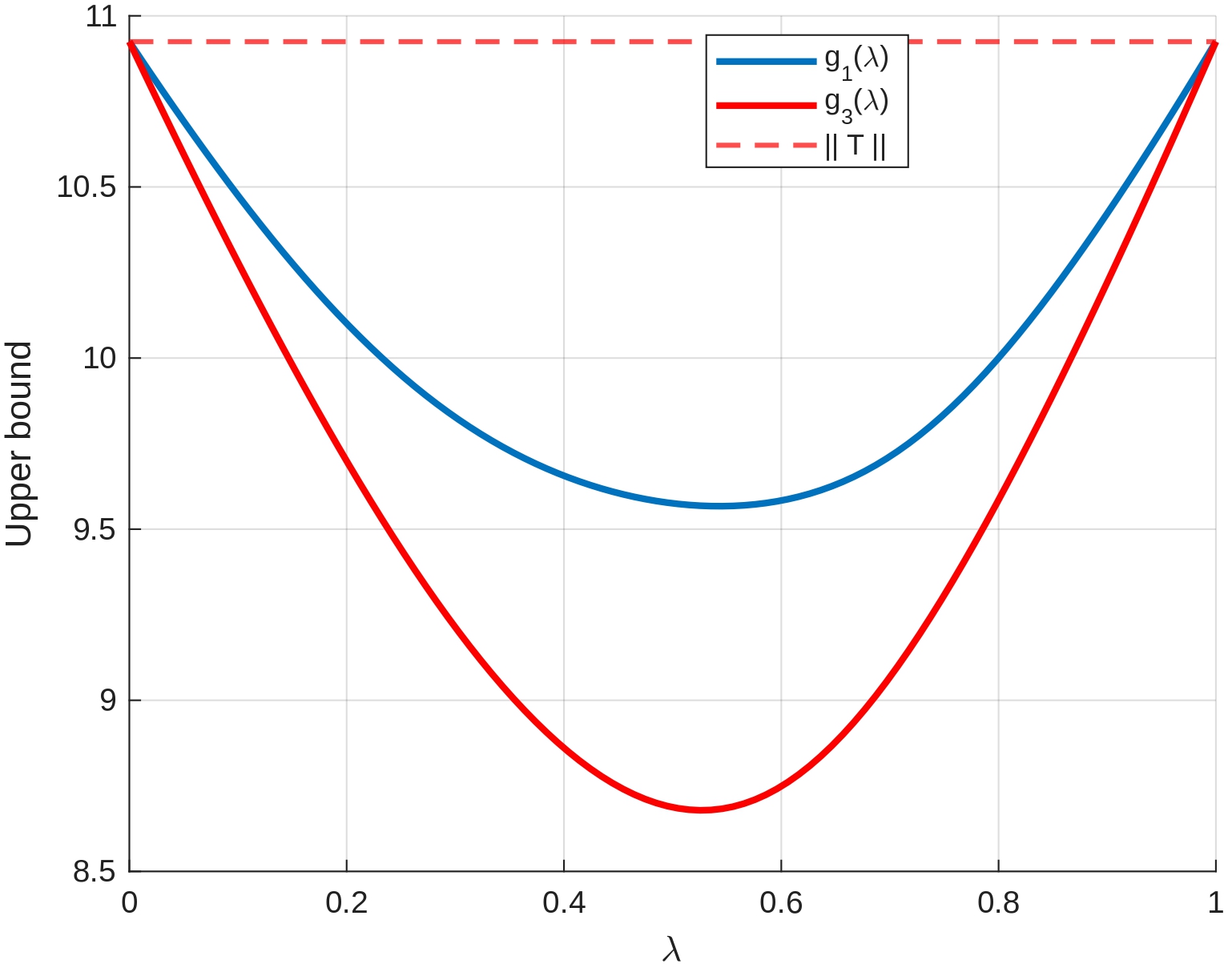}
        \caption{Upper bound for (\ref{26})}
        \label{fig:first1}
    \end{subfigure}
    \hfill
    \begin{subfigure}{0.41\textwidth}
        \centering
        \includegraphics[width=\linewidth]{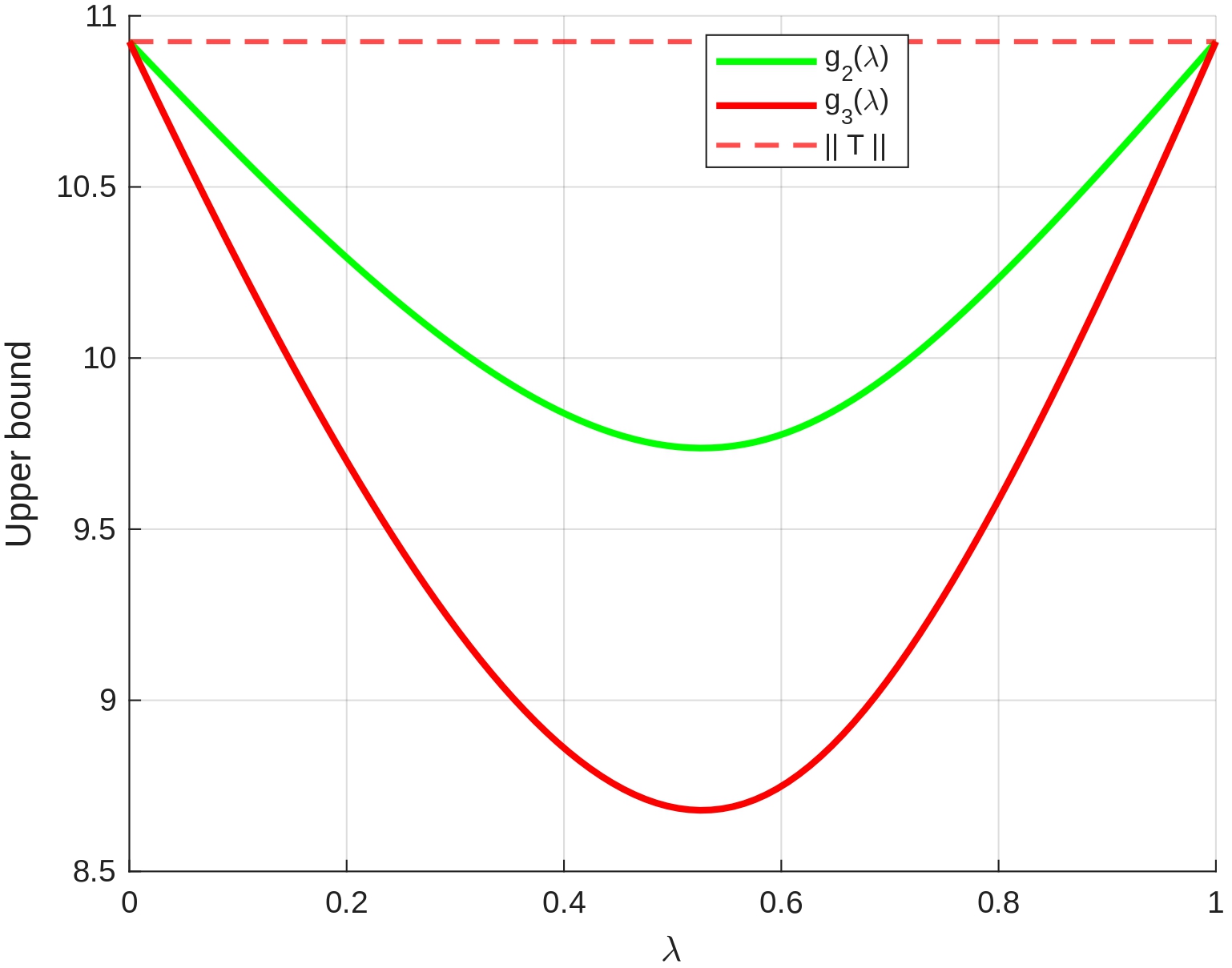}
        \caption{Upper bound for (\ref{27})}
        \label{fig:second1}
    \end{subfigure}
    \caption{Comparison of upper bounds of $\|M_{\lambda}(A)\|$ with $\|A\|.$}
    \label{fig:both1}
\end{figure}

In Fig.\ref{fig:both1},  we plot the upper bounds $g_1(\lambda)$ and $g_2(\lambda)$ obtained from inequalities (\ref{26}) and (\ref{27}), respectively, the exact quantity $g_3(\lambda)$ and the classical bound $\|A\|$. We note that $g_1(\lambda)$ and $g_2(\lambda)$ are strictly less than $\|A\|$ for $0<\lambda<1$, and all three bounds coincide at the endpoints $\lambda=0$ and $\lambda=1$.

Thus the upper bounds for $\|M_{\lambda}(A)\|$ derived from (\ref{26}) and (\ref{27}) is  sharper than the classical bound (\ref{10}).
The numerical example clearly demonstrates the efficacy of the new inequalities and confirms that the inequalities (\ref{26}) and (\ref{27}) are better than the standard estimate (\ref{10}).

\end{remark}

We now present the following result which  generalizes and improves the second inequality of (\ref{12}).

\begin{theorem}\label{Theorem5}
Let $T \in \mathcal{B}(\mathcal{H})$. Then for $\lambda \in [0,1]$, we have
\begin{align}\label{28}
\|M_{\lambda}(T)\|
&\le
\sqrt{
\|
\,\lambda^2 |T| + (1-\lambda)^2 |T^D|
\,\| \;
\|\,
\lambda^2 |T^*| + (1-\lambda)^2 |(T^D)^*|
\,\|
} \notag \\
&\quad
+
\lambda(1-\lambda)
\sqrt{
\,
\|\,|T|+|T^D|\,\|
\;
\|\,|T^*|+|(T^D)^*|\,\|
}\tag{28}
\end{align}

In particular,
\begin{align}\label{29}
\|\widehat{T}\|
&\le\frac12\sqrt{\,\|\,|T|+|T^D|\,\|\;\|\,|T^*|+|(T^D)^*|\,\|}.\tag{29}
\end{align}
\end{theorem}

\begin{proof}: We easily get the  inequality (\ref{29}) from (\ref{28}) by taking $\lambda=\frac{1}{2}.$

For any $x,y\in\mathcal{H}$ with $\|x\|=\|y\|=1$ and 
$\lambda\in[0,1]$, we have
\begin{align*}
|\langle M_\lambda(T)x,y\rangle|
&\le
\lambda |\langle Tx,y\rangle|
+(1-\lambda)|\langle T^{D}x,y\rangle|.
\end{align*}

Rewrite the right-hand side as
\begin{align*}
&= \lambda^2 |\langle Tx,y\rangle|
+(1-\lambda)^2 |\langle T^{D}x,y\rangle|  
+\lambda(1-\lambda)
\big(
|\langle Tx,y\rangle|
+|\langle T^{D}x,y\rangle|
\big).
\end{align*}

Using Lemma \ref{Lemma9}, we obtain 
\begin{align}\label{30}
&|\langle M_\lambda(T)x,y\rangle|\le
\lambda^2
\langle |T|x,x\rangle^{1/2}\, \langle |T^*|y,y\rangle^{1/2}
+
(1-\lambda)^2
\langle |T^{D}|x,x\rangle^{1/2}\, \langle |(T^D)^*|y,y\rangle^{1/2}\notag
\\
&\quad
+\lambda(1-\lambda)
\Big(
\langle |T|x,x\rangle^{1/2}\, \langle |T^*|y,y\rangle^{1/2}+\langle |T^{D}|x,x\rangle^{1/2}\, \langle |(T^D)^*|y,y\rangle^{1/2}
\Big).\tag{30}
\end{align}

Applying the fundamental fact   $(ab+cd)^2\le (a^2+c^2)\, (b^2+d^2)$ for any positive real numbers $a,b,c,d$ on the last inequality (\ref{30}),
we obtain
\begin{align*}
&|\langle M_\lambda(T)x,y\rangle|
\le\sqrt{\langle (\lambda^2|T|+(1-\lambda)^2|T^{D}|)x,x\rangle \; \langle (\lambda^2|T^*|+(1-\lambda)^2|(T^{D})^*|)y,y\rangle}
\\&+\lambda(1-\lambda)
\sqrt{
\langle (|T|+|T^{D}|)x,x\rangle\; \langle (|T^*|+|(T^{D})^*|)y,y\rangle
}.
\end{align*}

Taking supremum over all $x,y\in\mathcal{H}$ with 
$\|x\|=\|y\|=1$, we get the required inequality (\ref{28}).

\end{proof}

\begin{remark}

 By applying AM-GM inequality in (\ref{29}), we infer that 
   \begin{align*}
\|\widehat{T}\|
&\le \frac12
\sqrt{
\,
\|\,|T|+|T^D|\,\|
\;
\|\,|T^*|+|(T^D)^*|\,\|
} \\&\le\frac{
\|\,|T|+|T^D|\,\|
+ \|\,|T^*|+|(T^D)^*|\,\|
}{4}\le\|T\|.\end{align*}
\begin{example}
To show  proper improvement, we consider a numerical example.

Let 
$T=\begin{bmatrix}
   0 & 0 & 7\\
   0 & 0 & 5\\
   0 & 0 & 9
\end{bmatrix}.$ By using MATLAB, we get 

\begin{align*}\|\widehat{T}\|\approx 9.974969 
&<\frac12
\sqrt{
\,
\|\,|T|+|T^D|\,\|
\;
\|\,|T^*|+|(T^D)^*|\,\|}\approx 9.978785\\& < \frac{
\|\,|T|+|T^D|\,\|
+ \|\,|T^*|+|(T^D)^*|\,\|
}{4}\approx 10.004741 < \|T\| \approx 12.449900.
\end{align*}
\end{example}
  
\end{remark}

\begin{remark}
From (\ref{28}), we see that
\begin{align*}
\|M_{\lambda}(T)\|
&\le
\sqrt{
\|
\,\lambda^2 |T| + (1-\lambda)^2 |T^D|
\,\|\;
\|\,
\lambda^2 |T^*| + (1-\lambda)^2 |(T^D)^*|
\,\|
}  \notag \\
&\quad
+
\lambda(1-\lambda)
\sqrt{
\,
\|\,|T|+|T^D|\,\|
\;
\|\,|T^*|+|(T^D)^*|\,\|
}.
\end{align*}
By the triangle inequality, we get 
\begin{align*}
 \|M_{\lambda}(T)\|&\le\sqrt{\left(\lambda^2\|\,|T|\,\|+(1-\lambda)^2\|\,|T^D|\,\|\right)\;\left(\lambda^2\|\,|T^*|\,\|+(1-\lambda)^2\|\,|(T^D)^*|\,\|\right)}\\&+\lambda(1-\lambda)
\sqrt{
\,
\left(\|\,|T|\,\|+\|\,|T^D|\,\|\right)
\;
\left(\|\,|T^*|\,\|+\|\,|(T^D)^*|\,\|\right)
}.    
\end{align*}
By using the facts for any operator $S \in \mathcal{B}(\mathcal{H}),$
$\|\,|S|\,\|=\|\,|S^*|\,\|=\|S\|$  and $\|S^D\|\le\|S\|,$ we obtain
\begin{align*}
   \|M_{\lambda}(T)\|\le(2\lambda^2-2\lambda+1)\|T\|+2\lambda\|T\|-2\lambda^2\|T\|=\|T\|.
\end{align*}
For comparison, we consider a $3\times 3$ matrix 
$B=\begin{bmatrix}
 3 - 5i&-4 + 3i&6 + 4i\\
 1 - 2i&-2 + 5i&3 + 5i\\
 5 - 2i&-5 - 11i&-2 + 9i
\end{bmatrix}.$

Let
\begin{align*}h_1(\lambda)&=\sqrt{
\|\lambda^2|B|+(1-\lambda)^2|B^D|\,\|\; \|\lambda^2|B^*|+(1-\lambda)^2|(B^D)^*|\,\|}
\\&+\lambda(1-\lambda)\sqrt{\|\,|B|+|B^D|\,\|\;\|\,|B^*|+|(B^D)^*|\,\|
}\end{align*} and $h_2(\lambda)=\|M_{\lambda}(B)\|.$

\begin{figure}[ht]
    \centering
     
\includegraphics[width=0.45\textwidth]{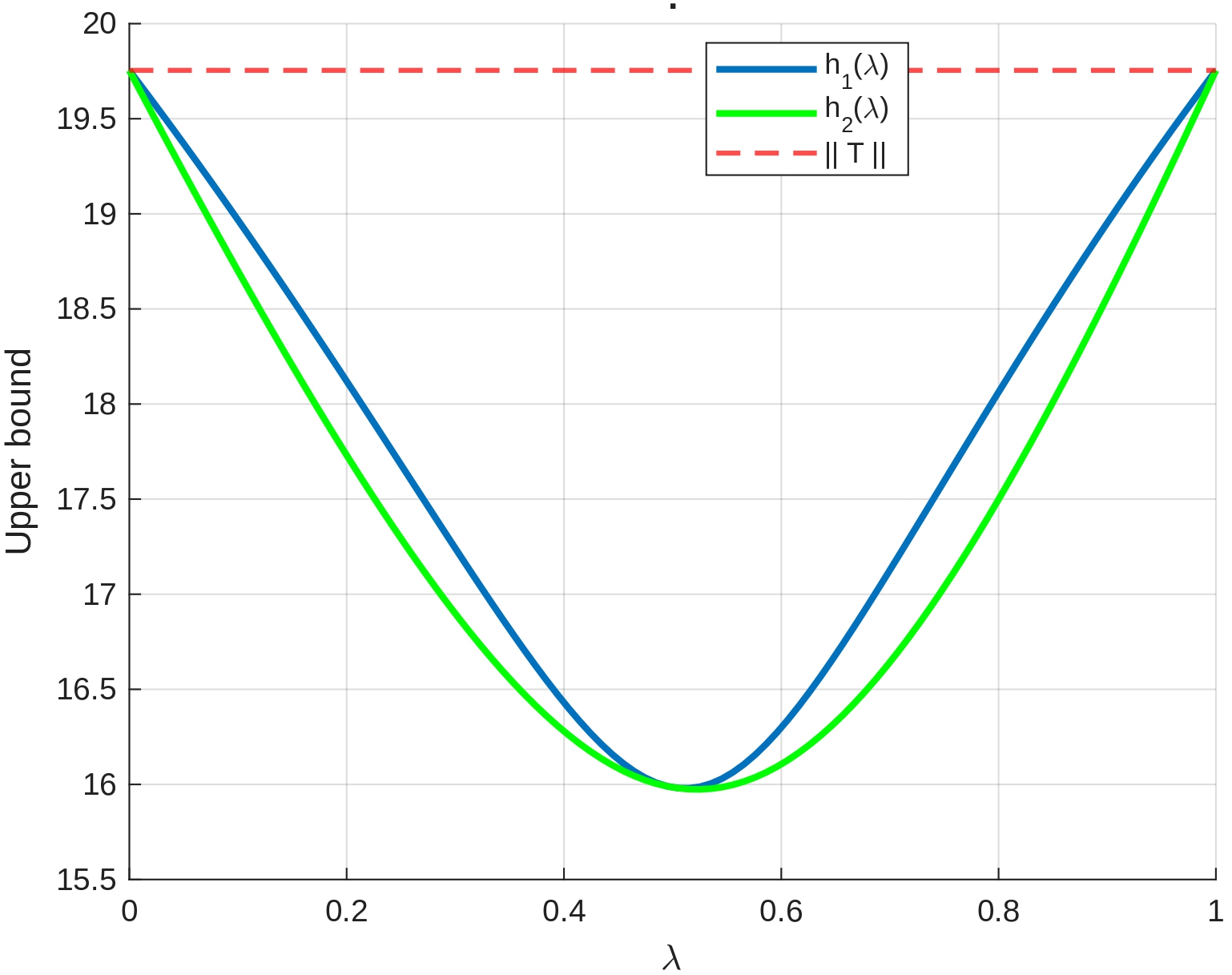}
        \caption{Upper bound for (\ref{28})}
        
    \label{fig:both}
\end{figure}
  
The upper bound $h_{1}(\lambda)$ obtained from inequality (\ref{28}) is compared with the exact quantity $h_{2}(\lambda)$ and the constant bound $\|B\|$ obtained from (\ref{10}) in  Fig.\ref{fig:both}. It is shown that $h_1(\lambda) < \|B\|$ for every $\lambda\in(0,1)$ and the bounds coincide at the endpoints $\lambda=0$ and $\lambda=1$. Moreover, the graph of $h_1(\lambda)$ closely follows the graph of $h_2(\lambda)$ and attains its minimum near $\lambda=\frac12$.

Thus the upper bound for $\|M_{\lambda}(B)\|$ derived from (\ref{28}) is much sharper than the traditional bound (\ref{10}). The numerical example shows clearly the usefulness of the new inequality and verifies that (\ref{28}) is an improvement of (\ref{10}).

\end{remark}

Next, we obtain an upper bound for the norm of the $\lambda$-mean transform $M_{\lambda}(T),$  which provides a quantitative link between $M_{\lambda}(T)$ and the generalized Aluthge transform $\widehat{T}(t)$.

\begin{theorem}\label{Theorem6}
Let $T = U|T|$ be the polar decomposition of $T$. Then for $\lambda,t \in [0,1]$
\begin{equation}\label{31}
\|M_\lambda(T)\|
\le
\frac{\max\{\lambda,1-\lambda\}}{2}
\left(
\||T|^{2t} + |T|^{2(1-t)}\|
\right)
+ \sqrt{\lambda(1-\lambda)}\, \|\widehat{T}(t)\|.
\tag{31}\end{equation}

In particular (for $\lambda=t = \dfrac12$),
\begin{equation}\label{32}
\|\widehat{T}\|
\le
\frac12 \|T\| + \frac12 \|\widetilde{T}\|
\le
\|T\|.
\tag{32}\end{equation}
\end{theorem}

\begin{proof}: We easily get the  inequality (\ref{32}) from (\ref{31}) by taking $\lambda=t=\frac{1}{2}.$
Therefore, it is enough to prove the inequality (\ref{31}).

Suppose that
$
A=
\begin{bmatrix}
\sqrt{\lambda}\,U|T|^{t} & \sqrt{1-\lambda}\,|T|^{1-t} \\
0 & 0
\end{bmatrix}
\quad\text{and}\quad
B^{*}=
\begin{bmatrix}
\sqrt{\lambda}\,|T|^{1-t} & 0 \\
\sqrt{1-\lambda}\,|T|^{t}U & 0
\end{bmatrix},
$
where $\lambda\in[0,1]$ and $t\in[0,1]$.

We have, 
\begin{align*}
 \|M_{\lambda}(T)\|
&=
\|\lambda U|T|+(1-\lambda)|T|U\|
\\&=
\|AB^{*}\|.   
\end{align*}
By applying Lemma \ref{Lemma8}, we get 
\begin{align*}
\|M_{\lambda}(T)\|
&\le
\frac12
\left\|
\begin{bmatrix}
\sqrt{\lambda}\,|T|^{t}U^{*} & 0 \\
\sqrt{1-\lambda}\,|T|^{1-t} & 0
\end{bmatrix}
\begin{bmatrix}
\sqrt{\lambda}\,U|T|^{t} & \sqrt{1-\lambda}\,|T|^{1-t} \\
0 & 0
\end{bmatrix}
\right. \\
&\qquad\left.
+
\begin{bmatrix}
\sqrt{\lambda}\,|T|^{1-t} & 0 \\
\sqrt{1-\lambda}\,|T|^{t}U & 0
\end{bmatrix}
\begin{bmatrix}
\sqrt{\lambda}\,|T|^{1-t} & \sqrt{1-\lambda}\,U^{*} |T|^{t}\\
0 & 0
\end{bmatrix}
\right\|\\&=
\frac12
\left\|
\begin{bmatrix}
\lambda \,\left(|T|^{2t} +|T|^{2(1-t)}\right) & 2\sqrt{\lambda(1-\lambda)}\,\left(\widehat{T}(t)\right)^{*} \\
2\sqrt{\lambda(1-\lambda)}\widehat{T}(t) & (1-\lambda)\,\left(|T|^{2t} + |T|^{2(1-t)}\right)
\end{bmatrix}
\right\| 
\\&\le
\frac12
\left\|
\begin{bmatrix}
\lambda \,\left(|T|^{2t} +|T|^{2(1-t)}\right) & 0 \\
0 & (1-\lambda)\,\left(|T|^{2t} + |T|^{2(1-t)}\right)
\end{bmatrix}
\right\|\\&+
\sqrt{\lambda(1-\lambda)}
\left\|\begin{bmatrix}
0 & (\widehat{T}(t))^{*} \\
\widehat{T}(t) & 0
\end{bmatrix}
\right\|\\&=\frac{\max\{\lambda,1-\lambda\}}{2}
\|\,|T|^{2t}+|T|^{2(1-t)}\| 
+\sqrt{\lambda(1-\lambda)}\,\|\widehat{T}(t)\|.
\end{align*}
The last equality is obtained from the following fact that  $\left\|\begin{bmatrix}
X & 0\\
0 & Y
\end{bmatrix}\right\|=\left\|\begin{bmatrix}
0 & X\\
Y & 0
\end{bmatrix}\right\|=\max\{\|X\|, \|Y\|\}$     for  $X,Y\in \mathcal{B}(\mathcal{H}).$
\end{proof}
\begin{example}
We consider an example to show proper improvement for the inequality (\ref{32}). 

Let $T=\begin{bmatrix}
    2+5i & 3+7i & 1+3i \\
    3+2i & 2+3i & 7+3i \\
    1+5i & 7+5i & 3+11i
\end{bmatrix}.$ By using MATLAB, we obtain  $\|\widehat{T}\|\approx 18.427622,$  $\|\widetilde{T}\|\approx 18.568053$ and  $\|T\|\approx 19.170412.$
 It shows that \begin{align*} \|\widehat{T}\|\approx 18.427622< \frac12 \|T\| + \frac12 \|\widetilde{T}\|\approx 18.568053   < \|T\| \approx 19.170412.\end{align*}
\end{example}
The following characterization has been presented in  [\citenum{altwaijry2024new}, Theorem 4.1]; nonetheless, our derivation is conceptually distinct.  As an application of our investigation that we mentioned earlier, offer a simpler and structurally informed approach to the classical characterization that we intend to discuss.
\begin{theorem}\label{Theorem7}
Let $T \in \mathcal{B}(\mathcal{H})$. Then the following conditions are equivalent:
\begin{enumerate}
    \item[(i)] $\|\widehat{T}\|=\|T\|$;
    \item[(ii)] $\|\widetilde{T}\|=\|T\|$;
    \item[(iii)] $\|T\|^{2}=\|T^{2}\|$.
\end{enumerate}
\end{theorem}

\begin{proof}: $(i)\Rightarrow(ii)$ Suppose that $\|\widehat{T}\|=\|T\|$. Then  from (\ref{32}), we directly get $\|T\|\le \|\widetilde{T}\|.$ We always have $\|\widetilde{T}\|\le \|T\|.$ Hence, $\|\widetilde{T}\|=\|T\|.$

$(ii)\Rightarrow(iii)$
Suppose that $\|\widetilde{T}\|=\|T\|$. Then
\begin{align*}
\|T\|^{2}
&=\|\widetilde{T}\|^{2}
=\|\widetilde{T}\widetilde{T}^{*}\|  \\
&=\|(|T|^{1/2}U|T|^{1/2})(|T|^{1/2}U|T|^{1/2})^{*}\| \\
&=\||T|^{1/2}U|T|U^{*}|T|^{1/2}\| \\
&=\||T|^{1/2}|T^{*}||T|^{1/2}\|.\end{align*}

Using the facts that $r(.)$ is commutative and  for a self-adjoint operator $Y$, we have $r_{\sigma}(Y)=\|Y\|$, and $r_{\sigma}(Y)\le\|Y\|$ for any bounded linear operator $Y$, we get 

\begin{align*}
 \|T\|^{2}
=r_{\sigma}(|T|^{1/2}|T^{*}||T|^{1/2})
=r_{\sigma}(|T||T^{*}|)
\le \||T||T^{*}|\|
=\|T^2\|\le\|T\|^{2}.   
\end{align*}

So, $\|T\|^{2}=\|T^{2}\|.$

$(iii)\Rightarrow(i)$ Assume that $\|T\|^{2}=\|T^{2}\|$. Now,
\begin{align*}
\|T\|^{2}
&=\|T^{2}\|
=\||T||T^{*}|\|  \\
&=\||T|U|T|U^{*}\|
=\||T|^{1/2}(|T|^{1/2}U|T|^{1/2})|T|^{1/2}U^*\| \\
&\le
\||T|^{1/2}\|
\,\||T|^{1/2}U|T|^{1/2}\|
\,\||T|^{1/2}U^*\| \\
&=\|T\|\,
\|\widehat{T}\|.
\end{align*}
By using the fact $\|A^{\frac{1}{2}}XB^{\frac{1}{2}}\|\le \frac{1}{2}\|AX+XB\|$ for $X\in \mathcal{B}(\mathcal{H})$ and $A,B\in \mathcal{B}(\mathcal{H})^{+},$ we obtain 
\begin{align*}
\|T\|^2\le \|T\|\, \|\frac{1}{2} U|T|+\frac{1}{2} |T|U\|=\|T\|\,
\|\widehat{T}\|.    
\end{align*}
Since always $\|\widehat{T}\|\le \|T\|$, we conclude that
$\|T\|^{2}\le \|T\|\,
\|\widehat{T}\|\le \|T\|^2.$
Hence, $\|\widehat{T}\|= \|T\|.$
\end{proof}

Through the utilization of an integral numerical-radius representation of an associated operator matrix, the subsequent theorem offers a revised norm estimate for the $\lambda$-mean transform. This provides a more distinct relationship between $\|M_{\lambda}(T)\|,$ $r_{\omega}(T),$ $\|T\|$ and $\|T^{D}\|.$ It also refines the inequality (\ref{9}).

\begin{theorem}\label{Theorem8}
 Let $T\in \mathcal{B}(\mathcal{H})$. Then for $\lambda\in [0,1]$, we have
\begin{equation}\begin{aligned}\label{33}
\|M_{\lambda}(T)\|
&\le
2\int_{0}^{1}
r_{\omega}
\left(
\begin{bmatrix}
0 &
\lambda tT +(1-\lambda)(1-t)T^{D}
\\
(1-\lambda)t(T^{D})^{*}
+\lambda(1-t)T^{*}
&
0
\end{bmatrix}
\right)dt
\\&\le
2\,r_{\omega}
\left(
\begin{bmatrix}
0 & \lambda T
\\
(1-\lambda)(T^{D})^{*} & 0
\end{bmatrix}
\right)\\&\le\sqrt{\left\Vert\lambda^2 \vert{}T\vert{}^2 + (1-\lambda)^2 \vert{}T^D\vert{}^2\right\Vert + 2\lambda(1-\lambda) r_{\omega}\left(  T^* T^D \right)}
\\&\le
\lambda \|T\|
+
(1-\lambda)\|T^{D}\|.\end{aligned}\tag{33}
 \end{equation}  
\end{theorem}

\begin{proof}:
For any $\lambda\in [0,1]$, we have
$$ M_{\lambda}(T)
=
\lambda T +(1-\lambda)T^{D},
$$
where $T=U|T|$ is the polar decomposition of $T$ and $T^{D}=|T|U.$

Since the operator norm of a self-adjoint off-diagonal matrix $\begin{bmatrix} 0 & L \\ L^* & 0 \end{bmatrix}$ is simply $\Vert{}L\Vert{}$, taking $L = \lambda T +(1-\lambda)T^{D}$ yields:
$$
\|M_{\lambda}(T)\|
=\|\lambda T +(1-\lambda)T^{D}\|=
\|P+P^{*}\|,
$$
where
\[
P=
\begin{bmatrix}
0 & \lambda T
\\
(1-\lambda)(T^{D})^{*} & 0
\end{bmatrix}.
\]

By using the basic result
$
|a+b|
\le
2\int_{0}^{1}|ta+(1-t)b|\,dt
\le
|a|+|b|,
$
where $a,b\in \mathbb{C}$, one can easily obtain that
$$
r_{\omega}(S+S^{*})
\le
2\int_{0}^{1}
r_{\omega}\bigl(tS+(1-t)S^{*}\bigr)\,dt
\le
2r_{\omega}(S).
$$

So,
\begin{align*}
\|M_{\lambda}(T)\|
&=r_{\omega}(P+P^{*})
\\
&\le
2\int_{0}^{1}
r_{w}
\left(
\begin{bmatrix}
0 &
\lambda tT +(1-\lambda)(1-t)T^{D}
\\[0.2cm]
(1-\lambda)t(T^{D})^{*}
+\lambda(1-t)T^{*}
&
0
\end{bmatrix}
\right)dt
\\
&\le
2\,r_{w}
\left(
\begin{bmatrix}
0 & \lambda T
\\[0.2cm]
(1-\lambda)(T^{D})^{*} & 0
\end{bmatrix}
\right).
\end{align*}

A well-known identity for the numerical radius of off-diagonal block matrices for any $A, B \in \mathcal{B}(\mathcal{H})$ is:  $$r_{\omega} \left( \begin{bmatrix} 0 & X \\ Y^* & 0 \end{bmatrix} \right) = \frac{1}{2} \sup_{\theta \in \mathbb{R}} \left\Vert{} e^{i\theta} X + e^{-i\theta} Y \right\Vert.$$

By setting $X = \lambda T$ and $Y = (1-\lambda) T^D$, we immediately get

\begin{align*}
\left\Vert M_{\lambda} (T)\right\Vert &\le   \sup_{\theta \in \mathbb{R}} \left\Vert{} e^{i\theta} (\lambda T) + e^{-i\theta}(1-\lambda) T^{D} \right\Vert{}\\&=\sup_{\theta \in \mathbb{R}}\sqrt{\left\Vert\lambda^2 \vert{}T\vert{}^2 + (1-\lambda)^2 \vert{}T^D\vert{}^2 + 2\lambda(1-\lambda) \operatorname{Re}\left( e^{-2i\theta} T^* T^D \right)\right\Vert}\\& \le \sup_{\theta \in \mathbb{R}}\sqrt{\left\Vert\lambda^2 \vert{}T\vert{}^2 + (1-\lambda)^2 \vert{}T^D\vert{}^2\right\Vert + 2\lambda(1-\lambda) \left\Vert\operatorname{Re}\left( e^{-2i\theta} T^* T^D \right)\right\Vert}.
\end{align*}

By using the result: $r_{\omega}(X)=\sup_{\theta \in \mathbb{R}}\left\Vert\operatorname{Re}\left( e^{i\theta} T^* T^D \right)\right\Vert$ for any $X \in \mathcal{B}(\mathcal{H}),$ we obtain the third inequality. 

To establish the fourth inequality, we use the result $r_{\omega}(XY)\le \|XY\|\le\|X\| \|Y\|$ for any $X, Y\in \mathcal{B}(\mathcal{H})$.

\end{proof}

For $i=1,2,\ldots,n$, let $\mathcal{H}_i$ stand for a complex Hilbert space with the inner product $\langle .,. \rangle_i$.  Let $\mathbb{B}(\mathcal{H}_j, \mathcal{H}_i)$ ($1\leq i,j \leq n$) denote the  spaces of all bounded linear operators from $\mathcal{H}_j$ to $\mathcal{H}_i$, equipped with the operator norm topology. We define $\mathbb{H}=\bigoplus_{i=1}^n \mathcal{H}_i$. Consequently, $\mathbb{H}$ is a Hilbert space.

For $A, B, C, D \in \mathcal{B}(\mathcal{H})$, the $2 \times 2$ operator matrix
$\begin{bmatrix}
A & B \\
C & D
\end{bmatrix}
\in \mathcal{B}(\mathcal{H} \oplus \mathcal{H}),
$
and is defined by
$\begin{bmatrix}
A & B \\
C & D
\end{bmatrix}
x
=
\begin{pmatrix}
A x_1 + B x_2 \\
C x_1 + D x_2
\end{pmatrix},
\quad \forall x =
\begin{pmatrix}
x_1 \\
x_2
\end{pmatrix}
\in \mathcal{H} \oplus \mathcal{H}.
$

The subsequent theorem provides a computable upper bound for $\|M_{\lambda}(S)\|$ where $S$ is an off-diagonal operator matrix, therefore connecting the norm of the transformation to the norm values of its component operators.
\begin{theorem}\label{Theorem9}
 Let
$S=
\begin{bmatrix}
0 & A \\
B & 0
\end{bmatrix}
\in \mathcal{B}(\mathcal{H}\oplus\mathcal{H}).
$
Then for any $\lambda\in[0,1]$, we have
\begin{align}\label{34}
\|M_{\lambda}(S)\|
\le
\max
\Big\{
\|\lambda |A^{*}|+(1-\lambda)|B|\|,
\,
\|\lambda |B^{*}|+(1-\lambda)|A|\|
\Big\}.
\tag{34}\end{align}  
\end{theorem}

\begin{proof}:
For any $\lambda\in[0,1]$, we have
\[
M_{\lambda}(T)
=
\lambda T+(1-\lambda)T^{D},
\]
where $T=U|T|$ is the polar decomposition of $T$ and
$T^{D}=|T|U$.

Using the fact that for $t>0$,
\[
U|T|^{t}=|T^*|^{t}U,
\]
we obtain
\begin{align*}
\|M_{\lambda}(T)\|
&=
\|\lambda U|T|+(1-\lambda)|T|U\|  \\
&=
\|(\lambda |T^*|+(1-\lambda)|T|)U\|  \\
&\le
\|\lambda |T^{*}|+(1-\lambda)|T|\|.
\end{align*}

 By simple calculations, we get
\[
|S|=
\begin{bmatrix}
|B| & 0 \\
0 & |A|
\end{bmatrix},
\qquad
|S^{*}|=
\begin{bmatrix}
|A^*| & 0 \\
0 & |B^*|
\end{bmatrix}.
\]

Hence,
\begin{equation*}
    \|\lambda |S^{*}|+(1-\lambda)|S|\|=
\left\|
\begin{bmatrix}
\lambda |A^*|+(1-\lambda)|B| & 0 \\
0 & \lambda |B^*|+(1-\lambda)|A|
\end{bmatrix}
\right\|.
\end{equation*}
Since the operator norm of a block diagonal matrix $\begin{bmatrix} X & 0 \\ 0 & Y \end{bmatrix}$ is $\max\{\|X\|, \|Y\|$, the required inequality (\ref{34}) follows immediately.

\end{proof}
Let $T=U|T|$ be the polar decomposition of $T$ and 
$T^{D}=|T|U$.
Setting $A = B = T$, the operator $S$ reduces to

$ S= \begin{bmatrix} 0 & T \\ T & 0 \end{bmatrix}=  \begin{bmatrix} 0 & U \\ U & 0 \end{bmatrix} \begin{bmatrix} |T| & 0 \\ 0 & |T| \end{bmatrix}$
 \text{and} $S^{D} = \begin{bmatrix} |T| & 0 \\ 0 & |T| \end{bmatrix}\begin{bmatrix} 0 & U \\ U & 0 \end{bmatrix}=\begin{bmatrix} 0 & T^{D} \\ T^{D} & 0 \end{bmatrix}.$

Utilizing the norm identity for off-diagonal block matrices, $\left\| \begin{bmatrix} 0 & X \\ X & 0 \end{bmatrix} \right\| = \|X\|$, we obtain $\|M_{\lambda}(S)\| = \|\lambda S + (1-\lambda) S^{D}\|=\left \|\begin{bmatrix} 0 &  M_{\lambda}(T)\\ M_{\lambda}(T) & 0 \end{bmatrix}\right \|= \| M_{\lambda}(T)\|$. 

Moreover, substituting $A = B = T$ into the right-hand side of inequality (34) yields
\[
\max\Big\{ \|\lambda |T^{*}| + (1-\lambda)|T|\|, \, \|\lambda |T^{*}| + (1-\lambda)|T|\| \Big\} = \|\lambda |T^{*}| + (1-\lambda)|T|\|.
\]

Next, we find a necessary and sufficient criterion for the norm-preserving characteristic of the $\lambda$-mean transform.

\begin{theorem}
Let $T\in \mathcal{B}(\mathcal{H})$. Then
\[
\|M_{\lambda}(T)\|=\|T\|
\quad \text{for } \lambda\in (0,1)
\]
if and only if
\[
\|T^{2}\|=\|T\|^{2}.
\]
\end{theorem}

\begin{proof}:
If we take $A=B=T$ in (\ref{34}), then  it is clear that
\begin{align*}
\|M_{\lambda}(T)\|
\le
\|\lambda |T^{*}|+(1-\lambda)|T|\|
\le
\lambda\|T\|+(1-\lambda)\||T|\|
=
\|T\|.
\end{align*}

Let us assume that
$\|M_{\lambda}(T)\|=\|T\|
\quad \text{for } \lambda\in [0,1].
$
Then
\[\|M_{\lambda}(T)\|=
\|T\|
=
\|\lambda |T^{*}|+(1-\lambda)|T|\|
\le
\|T\|.
\]
So,
\[
\|\lambda |T^{*}|+(1-\lambda)|T|\|
=
\|T\|
=
\|\lambda |T^{*}|\|
+
\|(1-\lambda)|T|\|.
\]

Using the fact that if $A,B$ are two positive semidefinite operators, then
$\|A+B\|=\|A\|+\|B\|$
if and only if
$\|AB\|=\|A\|\,\|B\|,$
we get
\[
\|\lambda(1-\lambda)|T^{*}||T|\|
=
\|\lambda| T^{*}|\|\,\|(1-\lambda)|T|\|.
\]
Since $\|\,|T^{*}||T|\,\|=\|T^{2}\|,$ we obtain $\|T^{2}\|=\|T\|^{2}.$

Conversely, let us assume that $\|T^{2}\|=\|T\|^{2}.$
That is,
\[\|\lambda(1-\lambda)|T^{*}||T|\|
=\|(1-\lambda)|T|\|\,\|\lambda| T^{*}|\|.\]

Since $\lambda |T^{*}|$ and $(1-\lambda)|T|$ are positive operators, we get
\[\|\lambda |T^{*}|\|+\|(1-\lambda)|T|\|=\|\lambda|T^{*}|+(1-\lambda)|T|\|.\]

  Therefore,
\begin{align*}
\|T\|
&=
\|(\lambda|T^{*}|U+(1-\lambda)|T|U)U^*\|\qquad\text{(since, $UU^*=I$)}
\\
&\le
\|\lambda|T^{*}|U+(1-\lambda)|T|U\|
\\
&=\|\lambda U|T|+(1-\lambda)|T|U\|\qquad\text{(since, $U|T|=|T^*|U$)}\\&
=\|M_{\lambda}(T)\|
\le
\|T\|.
\end{align*}
This completes the proof.
  
\end{proof}

\begin{remark}
    For $\lambda = 1$, we have  $ M_{1}(T) = T$,  making this case trivial. When $\lambda = o$,  the above corollary fails to hold in general. To see this, consider the operator $ T = \begin{bmatrix} 0 & 5 \\ 7 & 0 \end{bmatrix}$.
Then $ |T| = \begin{bmatrix} 7 & 0 \\ 0 & 5 \end{bmatrix}$. From the polar decomposition of $T=U|T|,$ we get $U=\begin{bmatrix} 0 & 1 \\ 1 & 0 \end{bmatrix}.$
Then $M_{0}(T) = T^{D} =|T|U= \begin{bmatrix} 0 & 7 \\ 5 & 0 \end{bmatrix}$ and hence $\|T^{D}\| = \|T\| = 7$. A simple calculation gives $\|T^2\|=\left\|\begin{bmatrix}
  35&0\\0&35  \end{bmatrix} \right\|=35 \ne 49= \left\|\begin{bmatrix}
  0&5\\7&0 \end{bmatrix} \right\|^2=\|T\|^2. $

\end{remark}

Although norm inequalities for the $\lambda$-mean transform of general operators are well known, precise bounds fitting the $2 \times 2$ off-diagonal block structure have not appeared before. We close this gap by stating these bounds in the following theorem.

\begin{theorem}\label{Theorem10}
 Let
$T=
\begin{bmatrix}
0 & A \\
B & 0
\end{bmatrix}
\in \mathcal{B}(\mathcal{H}\oplus\mathcal{H}).
$
Then for any $\lambda\in[0,1]$, we have
\begin{align}\label{35}
(a)\,\|M_{\lambda}(T)\|
&\le\max\Big\{ \max(\Vert{}A\Vert{}, \Vert{}B\Vert{})^t, \, \max(\Vert{}A\Vert{}, \Vert{}B\Vert{})^{1-t} \Big\}\notag\\&\times \max\Big\{ \big\Vert{}\lambda \vert{}A^*\vert{}^t + (1-\lambda)\vert{}B\vert{}^{1-t}\big\Vert{}, \, \big\Vert{}\lambda \vert{}B^*\vert{}^t + (1-\lambda)\vert{}A\vert{}^{1-t}\big\Vert{} \Big\}.
\tag{35}\end{align}

\begin{align}\label{36}
(b)\,\|M_{\lambda}(T)\|
&\le\sqrt{\begin{aligned}
     &\max\Big\{ \big\Vert{}\lambda \vert{}A^*\vert{}^{2t} + (1-\lambda)\vert{}B\vert{}^{2(1-t)}\big\Vert{}, \big\Vert{}\lambda \vert{}B^*\vert{}^{2t} + (1-\lambda)\vert{}A\vert{}^{2(1-t)}\big\Vert{} \Big\} \\
     &\times\left(\lambda \max(\Vert{}A\Vert{}, \Vert{}B\Vert{})^{2(1-t)} + (1-\lambda)\max(\Vert{}A\Vert{}, \Vert{}B\Vert{})^{2t}\right).
\end{aligned}}
\tag{36}\end{align}

\begin{align}\label{37}
(c)\,\|M_{\lambda}(T)\|
&\le\Big( \max \Big\{ \big\| \lambda^2 |B| + (1-\lambda)^2 |A^D| \big\|, \, \big\| \lambda^2 |A| + (1-\lambda)^2 |B^D| \big\| \Big\}\notag\\
&\qquad \times \max \Big\{ \big\| \lambda^2 |A^*| + (1-\lambda)^2 |(B^D)^*| \big\|, \, \big\| \lambda^2 |B^*| + (1-\lambda)^2 |(A^D)^*| \big\| \Big\} \Big)^{1/2} \notag\\
&\quad + \lambda(1-\lambda) \Big( \max \Big\{ \big\| |B| + |A^D| \big\|, \, \big\| |A| + |B^D| \big\| \Big\}\notag \\
&\qquad \times \max \Big\{ \big\| |A^*| + |(B^D)^*| \big\|, \, \big\| |B^*| + |(A^D)^*| \big\| \Big\} \Big)^{1/2}
\tag{37}.\end{align}

\begin{align}\label{38}
(d)\,&\|M_{\lambda}(T)\|
\le
\frac{\max\{\lambda, 1-\lambda\}}{2} \left( \max\left\{ \| B \|^{2t}, \| A \|^{2t} \right\} + \max\left\{ \| A \|^{2(1-t)}, \| B\|^{2(1-t)} \right\} \right) \notag\\
& \qquad+\frac{\sqrt{\lambda(1-\lambda)}}{2} \max \left\{ \| |B|^t |A^*|^{1-t} +  |B|^{1-t} |A^*|^t \|, \| |A|^t |B^*|^{1-t}  + |A|^{1-t} |B^*|^t \| \right\}.
\tag{38}\end{align}
\end{theorem}

\begin{proof}
Let  $A = U|A|$ and  $B = V|B|$ be the polar decompositions of  $A$ and  $B$, respectively.

For $T = \begin{bmatrix} 0 & A \\ B & 0 \end{bmatrix}=\begin{bmatrix} 0 & U \\ V & 0 \end{bmatrix} \begin{bmatrix} |B| & 0 \\ 0 & |A| \end{bmatrix},$ we have

\[|T| = \begin{bmatrix} |B| & 0 \\ 0 & |A| \end{bmatrix},\,  |T^*| = \begin{bmatrix} |A^*| & 0 \\ 0 & |B^*| \end{bmatrix},\, \vert{}T^D\vert{} = \begin{bmatrix} \vert{}A^D\vert{} & 0 \\ 0 & \vert{}B^D\vert{} \end{bmatrix}\] and \[\vert{}(T^D)^*\vert{} = \begin{bmatrix} \vert{}(B^D)^*\vert{} & 0 \\ 0 & \vert{}(A^D)^*\vert{} \end{bmatrix}, \]where $A^D = \vert{}A\vert{}U, \, B^D = \vert{}B\vert{}V.$
Now, we obtain 
\begin{align*}
 \hat{T}(t) &= \frac{|T|^t U |T|^{1-t} + |T|^{1-t} U |T|^t}{2}\\&= \frac{1}{2}\begin{bmatrix} 0 & |B|^t U |A|^{1-t} + |B|^{1-t} U |A|^t \\ |A|^t V |B|^{1-t} + |A|^{1-t} V |B|^t & 0 \end{bmatrix}.
\end{align*}

From the inequality (\ref{26}), we get that

$$\begin{aligned}
    \left\| M_\lambda \left( \begin{bmatrix} 0 & A \\ B & 0 \end{bmatrix} \right) \right\|& \le \max\left\{\left\| \begin{bmatrix} 0 & A \\ B & 0 \end{bmatrix}  \right\|^{t}, \left\|  \begin{bmatrix} 0 & A \\ B & 0 \end{bmatrix}  \right\|^{1-t} \right\}\\&\times \left\|\begin{bmatrix} \lambda \vert{}A^*\vert{}^t + (1-\lambda)\vert{}B\vert{}^{1-t} & 0 \\ 0 & \lambda \vert{}B^*\vert{}^t + (1-\lambda)\vert{}A\vert{}^{1-t} \end{bmatrix}\right\|\\
    &= \max\Big\{ \max(\Vert{}A\Vert{}, \Vert{}B\Vert{})^t, \, \max(\Vert{}A\Vert{}, \Vert{}B\Vert{})^{1-t} \Big\}\\&\times \max\Big\{ \big\Vert{}\lambda \vert{}A^*\vert{}^t + (1-\lambda)\vert{}B\vert{}^{1-t}\big\Vert{}, \, \big\Vert{}\lambda \vert{}B^*\vert{}^t + (1-\lambda)\vert{}A\vert{}^{1-t}\big\Vert{} \Big\}.
\end{aligned}$$

From the inequality (\ref{27}), we get that

$$\begin{aligned}
    &\left\| M_\lambda \left( \begin{bmatrix} 0 & A \\ B & 0 \end{bmatrix} \right) \right\| \le\sqrt{
  \begin{aligned} &\left\|
\begin{bmatrix} \lambda \vert{}A^*\vert{}^{2t} + (1-\lambda)\vert{}B\vert{}^{2(1-t)} & 0 \\ 0 & \lambda \vert{}B^*\vert{}^{2t} + (1-\lambda)\vert{}A\vert{}^{2(1-t)} \end{bmatrix}\right\|\\&\times\left(\lambda\left\| \begin{bmatrix} 0 & A \\ B & 0 \end{bmatrix}  \right\|^{2(1-t)}+(1-\lambda)\left\| \begin{bmatrix} 0 & A \\ B & 0 \end{bmatrix}  \right\|^{2t}\right)\end{aligned}}\\
& \qquad\qquad\qquad= \sqrt{\begin{aligned}
     &\max\Big\{ \big\Vert{}\lambda \vert{}A^*\vert{}^{2t} + (1-\lambda)\vert{}B\vert{}^{2(1-t)}\big\Vert{}, \big\Vert{}\lambda \vert{}B^*\vert{}^{2t} + (1-\lambda)\vert{}A\vert{}^{2(1-t)}\big\Vert{} \Big\} \\
     &\times\left(\lambda \max(\Vert{}A\Vert{}, \Vert{}B\Vert{})^{2(1-t)} + (1-\lambda)\max(\Vert{}A\Vert{}, \Vert{}B\Vert{})^{2t}\right). 
\end{aligned}}
\end{aligned}$$.

From the inequality (\ref{28}), we infer that 

$$\begin{aligned}
&\left\| M_\lambda \left( \begin{bmatrix} 0 & A \\ B & 0 \end{bmatrix} \right) \right\|
\le \sqrt{
  \begin{aligned}
    &\left\| \begin{bmatrix} \lambda^2 |B| + (1-\lambda)^2 |A^D| & 0 \\ 0 & \lambda^2 |A| + (1-\lambda)^2 |B^D| \end{bmatrix} \right\| \\
    &\quad \times \left\| \begin{bmatrix} \lambda^2 |A^*| + (1-\lambda)^2 |(B^D)^*| & 0 \\ 0 & \lambda^2 |B^*| + (1-\lambda)^2 |(A^D)^*| \end{bmatrix} \right\|
  \end{aligned}
} \\
&\qquad\qquad +\lambda(1-\lambda)\sqrt{\left\| \begin{bmatrix} |B| + |A^D| & 0 \\ 0 & |A| + |B^D| \end{bmatrix} \right\| \left\| \begin{bmatrix} |A^*| + |(B^D)^*| & 0 \\ 0 & |B^*| + |(A^D)^*| \end{bmatrix} \right\|}\\
&= \Big( \max \Big\{ \big\| \lambda^2 |B| + (1-\lambda)^2 |A^D| \big\|, \, \big\| \lambda^2 |A| + (1-\lambda)^2 |B^D| \big\| \Big\} \\
&\qquad \times \max \Big\{ \big\| \lambda^2 |A^*| + (1-\lambda)^2 |(B^D)^*| \big\|, \, \big\| \lambda^2 |B^*| + (1-\lambda)^2 |(A^D)^*| \big\| \Big\} \Big)^{1/2} \\
&\quad + \lambda(1-\lambda) \Big( \max \Big\{ \big\| |B| + |A^D| \big\|, \, \big\| |A| + |B^D| \big\| \Big\} \\
&\qquad \times \max \Big\{ \big\| |A^*| + |(B^D)^*| \big\|, \, \big\| |B^*| + |(A^D)^*| \big\| \Big\} \Big)^{\frac{1}{2}}.
\end{aligned}$$

From the inequality (31), we obtain

$$\begin{aligned}
&\left\| M_\lambda \left( \begin{bmatrix} 0 & A \\ B & 0 \end{bmatrix} \right) \right\| \\
&\le \frac{\max\{\lambda, 1 - \lambda\}}{2} \left( \max \left\{ \left\| |B|^{2t} \right\|, \left\| |A|^{2t} \right\| \right\} + \max \left\{ \left\| |A|^{2(1 - t)} \right\|, \left\| |B|^{2(1 - t)} \right\| \right\} \right) \\
& + \frac{\sqrt{\lambda(1 - \lambda)}}{2} \max \left\{ \left\| |B|^t U |A|^{1 - t} + |B|^{1 - t} U |A|^t \right\|, \left\| |A|^t V |B|^{1 - t} + |A|^{1 - t} V |B|^t \right\| \right\} \\
&\le \frac{\max\{\lambda, 1 - \lambda\}}{2} \left( \max \left\{ \|B\|^{2t}, \|A\|^{2t} \right\} + \max \left\{ \|A\|^{2(1 - t)}, \|B\|^{2(1 - t)} \right\} \right) \\
& + \frac{\sqrt{\lambda(1 - \lambda)}}{2} \max \left\{ \left\| |B|^t |A^*|^{1 - t} + |B|^{1 - t} |A^*|^t \right\|, \left\| |A|^t |B^*|^{1 - t} + |A|^{1 - t} |B^*|^t \right\| \right\},
\end{aligned}$$
where $U|A|^q U^* = |A^*|^q$ and $U^* |A^*|^q U = |A|^q$ and $U|A|^q = |A^*|^q U$ for $q > 0$.

\end{proof}

\begin{remark}
 Combining Theorem \ref{Theorem10} with the simple inequality $r_{\omega}(X) \le \Vert{}X\Vert{}$ we see that all the bounds in (\ref{35}–\ref{38}) are also bounds for $r_{\omega}(M_{\lambda}(T))$.   
\end{remark}

The following result, as an alternative version of Gelfand's spectral radius formula, indicates that $r_\sigma(T)$ can be obtained asymptotically by the transformation $M_\lambda$.

\begin{proposition} Let $\lambda \in (0,1)$ and $T \in B(H)$. Then $r_{\sigma}(T) = \lim\limits_{n \to \infty} \|M_\lambda(T^n)\|^{\frac{1}{n}}$.
\end{proposition}
\begin{proof}
From the inequality (\ref{10}), and by using the basic result for $n \in \mathbb{N},$ $r(T^n) = (r(T))^n$, we obtain
\[
2\sqrt{\lambda - \lambda^2}\,r_{\sigma}(T^n)\le \|M_\lambda(T^n)\| \le \|T^n\|
\]
\[
\Rightarrow 2^{\frac{1}{n}}(\lambda - \lambda^2)^{\frac{1}{2n}}r_{\sigma}(T) \le \|M_\lambda(T^n)\|^{\frac{1}{n}} \le \|T^n\|^{\frac{1}{n}}
\]

Taking $n \to \infty$, and by applying Gelfand's Spectral Radius Formula, i.e.,
$r_{\sigma}(T) = \lim\limits_{k \to \infty} \|T^k\|^{\frac{1}{k}}.$ we get
\[
r_{\sigma}(T) \le \lim_{n \to \infty} \|M_\lambda(T^n)\|^{\frac{1}{n}} \le r_{\sigma}(T)
\Rightarrow \lim_{n \to \infty} \|M_\lambda(T^n)\|^{\frac{1}{n}} = r_{\sigma}(T).
\]

\end{proof}

We now state some basic features of $r_{\omega}$ applied to off-diagonal operator matrices under the transformation $M_\lambda$.

\begin{theorem}
Let $X,Y \in \mathcal{B}(\mathcal{H})$. Then
\begin{enumerate}
    \item[(a)] $r_{\omega} \left( M_{\lambda}\left(\begin{bmatrix} O & X \\ Y & O \end{bmatrix} \right)\right) = r_{\omega} \left( M_{\lambda}\left( \begin{bmatrix} O & e^{-i\gamma} X \\ e^{i\gamma} Y & O \end{bmatrix}\right) \right)= r_{\omega} \left( M_{\lambda}\left( \begin{bmatrix} O &  X \\ e^{i\gamma} Y & O \end{bmatrix}\right) \right)$ for all $\gamma \in \mathbb{R}$.
    
    \item[(b)] $r_{\omega} \left(  M_{\lambda}\left(\begin{bmatrix} O & X \\ Y & O \end{bmatrix} \right)\right) = r_{\omega}\left(  M_{\lambda}\left(\begin{bmatrix} O & Y \\ X & O \end{bmatrix}\right) \right)$.
\end{enumerate}
\end{theorem}

\begin{proof} 

(a) Let $U = \begin{bmatrix} I & O \\ O & e^{-i\gamma} I \end{bmatrix}, \gamma\in \mathbb{R}$, where $I$ is the identity operator in $\mathcal{B}(\mathcal{H})$. Then $U$ is a unitary operator on $\mathcal{H} \oplus \mathcal{H}$ and
\[
U^* \begin{bmatrix} O & X \\ Y & O \end{bmatrix} U = \begin{bmatrix} I & O \\ O & e^{-i\gamma} I \end{bmatrix}^* \begin{bmatrix} O & X \\ Y & O \end{bmatrix} \begin{bmatrix} I & O \\ O & e^{-i\gamma} I \end{bmatrix} = \begin{bmatrix} O & e^{-i\gamma} X \\ e^{i\gamma} Y & O \end{bmatrix}.
\]

By Lemma \ref{Lemma11} (b), we have 
\[
M_{\lambda}\left(\begin{bmatrix} O & e^{-i\gamma} X \\ e^{i\gamma} Y & O \end{bmatrix}\right)  =M_{\lambda}\left(U^* \begin{bmatrix} O & X \\ Y & O \end{bmatrix} U\right)=U^*M_{\lambda}\left(\begin{bmatrix} O &  X \\ Y & O \end{bmatrix}\right)U.
\]
Since $r_{\omega} (\cdot)$ is a weakly unitarily invariant norm, the proof of the first equality of  (a) is complete. 

If we consider $U = \begin{bmatrix} I & O \\ O & e^\frac{-i\gamma}{2} I \end{bmatrix}, \gamma\in \mathbb{R}$, then
\[r_{\omega} \left( M_{\lambda}\left(\begin{bmatrix} O & X \\ Y & O \end{bmatrix} \right)\right) = r_{\omega} \left( M_{\lambda}\left( \begin{bmatrix} O & e^\frac{-i\gamma}{2} X \\ e^\frac{i\gamma}{2} Y & O \end{bmatrix}\right) \right).\]

By Lemma \ref{Lemma11} (a), we have $M_{\lambda}\left( \begin{bmatrix} O &  X \\ e^{i\gamma} Y & O \end{bmatrix}\right)=e^\frac{i\gamma}{2}\left( M_{\lambda}\left( \begin{bmatrix} O & e^\frac{-i\gamma}{2} X \\ e^\frac{i\gamma}{2} Y & O \end{bmatrix}\right) \right).$

Since $r_{\omega}(cT)= |c|r_{\omega}(T)$ for anuy $c\in\mathbb{R},$ the proof of the second equality of  (a) is complete. 

The proof of (b) is similar to (a)  with $U = \begin{bmatrix} O & I \\ I & O \end{bmatrix}$. This completes the proof.

\end{proof}

\textbf{Funding} The first author and second author express gratitude to the University Grants Commission (UGC), Government of India, for providing financial assistance in the form of senior research fellowship.

\textbf{Data availability} Not applicable.

\textbf{Declarations}

\textbf{Conflict of interest } There are no conflicting interests, according to the authors.

\bibliography{sn-bibliography}

@article{bhatia1993more,
  title={More matrix forms of the arithmetic-geometric mean inequality},
  author={Bhatia, Rajendra and Davis, Chandler},
  journal={SIAM Journal on Matrix Analysis and Applications},
  volume={14},
  number={1},
  pages={132--136},
  year={1993},
  publisher={SIAM}
}

@article{kittaneh1988notes,
  title={Notes on some inequalities for {H}ilbert space operators},
  author={Kittaneh, Fuad},
  journal={Publications of the Research Institute for Mathematical Sciences},
  volume={24},
  number={2},
  pages={283--293},
  year={1988},
  publisher={Research Institute forMathematical Sciences}
}

@article{bhunia2021furtherance,
  title={Furtherance of numerical radius inequalities of {H}ilbert space operators},
  author={Bhunia, Pintu and Paul, Kallol},
  journal={Archiv der Mathematik},
  volume={117},
  pages={537--546},
  year={2021},
  publisher={Springer}
}

@article{dragomir2008power,
  title={Power inequalities for the numerical radius of a product of two operators in {H}ilbert spaces},
  author={Dragomir, Sever S},
  journal={Research report collection},
  volume={11},
  number={4},
  year={2008},
  publisher={School of Communications and Informatics, Faculty of Engineering and Science~…}
}

@article{kittaneh2003numerical,
  title={A numerical radius inequality and an estimate for the numerical radius of the {F}robenius companion matrix},
  author={Kittaneh, Fuad},
  journal={Studia Mathematica},
  volume={158},
  number={1},
  pages={11--17},
  year={2003}
}

@article{kittaneh2005numerical,
  title={Numerical radius inequalities for {H}ilbert space operators},
  author={Kittaneh, Fuad},
  journal={Studia Mathematica},
  volume={168},
  number={1},
  pages={73--80},
  year={2005}
}

@article{el2007numerical,
  title={Numerical radius inequalities for {H}ilbert space operators. II},
  author={El-Haddad, Mohammad and Kittaneh, Fuad},
  journal={Studia Mathematica},
  volume={2},
  number={182},
  pages={133--140},
  year={2007}
}

@article{[abu2015upper],
  title={Upper and lower bounds for the numerical radius with an application to involution operators},
  author={Abu-Omar, Amer and Kittaneh, Fuad},
  journal={The Rocky Mountain Journal of Mathematics},
  volume={45},
  number={4},
  pages={1055--1065},
  year={2015},
  publisher={JSTOR}
}

@article{sababheh2025inner,
  title={Inner product inequalities with applications},
  author={Sababheh, Mohammad and Moradi, Hamid Reza and Sahoo, Satyajit},
  journal={Linear and Multilinear Algebra},
  volume={73},
  number={9},
  pages={2089--2102},
  year={2025},
  publisher={Taylor \& Francis}
}

@article{zamani2021extension,
  title={On an extension of operator transforms},
  author={Zamani, Ali},
  journal={Journal of Mathematical Analysis and Applications},
  volume={493},
  number={2},
  pages={124546},
  year={2021},
  publisher={Elsevier}
}

@article{dragomir2017buzano,
  title={A Buzano type inequality for two Hermitian forms and applications},
  author={Dragomir, Sever S},
  journal={Linear and Multilinear Algebra},
  volume={65},
  number={3},
  pages={514--525},
  year={2017},
  publisher={Taylor \& Francis}
}

@article{buzano1974generalizzazione,
  title={Generalizzazione della diseguaglianza di Cauchy-Schwarz},
  author={Buzano, Maria Luisa},
  journal={Rend. Sem. Mat. Univ. e Politech. Torino},
  volume={31},
  number={1971},
  pages={73},
  year={1974}
}

@article{altwaijry2024new,
  title={New results on some transforms of operators in hilbert spaces},
  author={Altwaijry, Najla and Conde, Cristian and Feki, Kais and Stankovi{\'c}, Hranislav},
  journal={Bulletin of the Brazilian Mathematical Society, New Series},
  volume={55},
  number={3},
  pages={42},
  year={2024},
  publisher={Springer Berlin Heidelberg Berlin/Heidelberg}
}

@article{chabbabi2019mean,
  title={The mean transform and the mean limit of an operator},
  author={Chabbabi, Fadil and Curto, Ra{\'u}l and Mbekhta, Mostafa},
  journal={Proceedings of the American Mathematical Society},
  volume={147},
  number={3},
  pages={1119--1133},
  year={2019}
}

\end{document}